\documentclass{siamltex}
\usepackage{amsfonts,amsmath,amssymb}
\usepackage{mathrsfs,mathtools}
\usepackage{enumerate}
\usepackage{xspace,mydef}
\usepackage{hyperref}
 \usepackage{esint}
\usepackage{float}
 \usepackage{graphicx}
\DeclareMathAlphabet{\mathpzc}{OT1}{pzc}{m}{it}

\renewcommand\sp{\mathop{\mathrm{sp}}\nolimits}
\def\CB{\mathcal{B}}
\def\CF{\mathcal{F}}
\def\CT{\mathcal{T}}
\def\hdel{\widehat{\delta}}%\def\VK{V^{\E}}
\def\CM{\mathcal{X}}
\def\CN{\mathcal{Y}}
\def\N{\mathbb{N}}
\def\O{\Omega}
\def\P{\mathbb{P}}
\def\PiK{\Pi^{\nabla,E}}
\def\PiN{\Pi^{\nabla}}
\def\Pio{\Pi^{E}}
\def\R{\mathbb{R}}
\def\Vh{V_h}
\def\VK{V^{E}_h}
\def\WK{\widetilde{V}_h^E}
\def\l{\lambda}
\def\g{\gamma}
\def\HusE{H^{1+s}(E)}

\newcommand{\vertiii}[1]{{\left\vert\kern-0.25ex\left\vert\kern-0.25ex\left\vert #1 
    \right\vert\kern-0.25ex\right\vert\kern-0.25ex\right\vert}}
\usepackage[usenames,dvipsnames]{color}

\numberwithin{equation}{section}

\title{VEM discretization allowing small edges  for  the reaction-convection-diffusion equation: source and spectral problems\thanks{FL  was partially supported by DIUBB through project 2120173 GI/C Universidad del B\'io-B\'io and 
ANID-Chile through FONDECYT project 11200529 (Chile).GR was supported by Universidad de Los Lagos Regular R02/21.}}

\author{Felipe Lepe\thanks{GIMNAP-Departamento de Matem\'atica, Universidad del B\'io-B\'io, Casilla 5-C, Concepci\'on, Chile. \texttt{flepe@ubiobio.cl}.}
\and
Gonzalo Rivera\thanks{Departamento de Ciencias Exactas, Universidad de Los Lagos, Osorno, Chile.
\texttt{gonzalo.rivera@ulagos.cl}}}

\date{Draft version of \today.}

\begin{document}

\maketitle
\begin{abstract}
In this paper we analyze a lowest order virtual element method for the load classic reaction-convection-diffusion
problem and the convection-diffusion spectral problem,  where the assumptions on the polygonal meshes allow to consider small edges for the polygons. Under well defined seminorms depending on a suitable stabilization for this geometrical approach, we derive the well posedness of the numerical scheme and error estimates for the load problem, whereas for the spectral problem we derive convergence and error estimates fo the eigenvalues and  eigenfunctions.  We report numerical tests to asses the performance of the small edges on our numerical method for both problems under consideration.
\end{abstract}

\begin{keywords}
virtual element methods a priori error estimates, small edges.
\end{keywords}

\begin{AMS}
49K20, % Problems involving partial differential equations
49M25, % Discrete approximations
65N12, % Stability and convergence of numerical methods
65N15,  % Error bounds
65N25, % Eigenvalue problems
65N30, % Finite elements, Rayleigh-Ritz and Galerkin methods, finite methods
%65N50. % Mesh generation and refinement
\end{AMS}

%%%%%%%%%%%%%%%%%%%%%%%%%%%%%%%%%%%%%%%%%%%%%%%%%%%%%%%%%%%%%%%%%%%%%%%%%%%%%%%%%%%%%%
\section{Introduction}
\label{sec:introduccion}
%%%%%%%%%%%%%%%%%%%%%%%%%%%%%%%%%%%%%%%%%%%%%%%%%%%%%%%%%%%%%%%%%%%%%%%%%%%%%%%%%%%%%%
Let  $\Omega \subset \mathbb{R}^2$ be an open,  bounded, and convex domain with polygonal boundary $\partial\Omega$. We are interested in the convection-diffusion problem
\begin{equation}
\label{eq:state_equation}
\nabla \cdot (-\kappa(\mathbf{x}) \nabla u)  + \mathbf{\vartheta}(\mathbf{x}) \cdot \nabla u +\gamma(\mathbf{x}) u = f \quad \textrm{in}~\Omega, \qquad u = 0 \quad \textrm{on}~\partial\Omega,
\end{equation}
where  $\kappa$ and $\gamma$ are smooth functions $\Omega\rightarrow\mathbb{R}$ with $\kappa(\mathbf{x})\geq\kappa_0>0$ for all $\mathbf{x}\in\Omega$ and  $\mathbf{\vartheta}$ is a smooth vector-valued function  $\Omega\rightarrow\mathbb{R}^2$.
%\GR{such that $\gamma-\dfrac{1}{2}\div \mathbf{\vartheta}\geq 0$ in $\O$.}

It is well know that \eqref{eq:state_equation} is a mathematical model that represents a physical phenomenon 
involving particles, concentrations, fluids, etc., that are transferred inside a physical system due two processes: convection and diffusion.  The eigenvalue problem associated to this system, and variations of it, has been analyzed in the nowadays in \cite{MR2845628,MR3133493}.

The virtual element method (VEM), introduced in \cite{BBCMMR2013} for the first time for the Laplacian operator has shown great accuracy on the approximation of the solutions of partial differential equations, together with important reduction on computational costs, compared with other classic methods. Moreover, since VEM allows different geometries on the meshes, it is possible to implement it with excellent results in problems where partial differential equations are stated in domains which are not suitable, for instance, for the finite element method.  

The developments and applications fo VEM are increasing day by day. The literature
of VEM is quite extensive, and it is possible to find results in several problems for fluid and solid mechanics, electromagnetism, eigenvalue problems, parabolic problems, and adaptive methods, where primal and mixed formulations
have appeared to approximate several problems. We can mention  \cite{MR4441208,  AABMR13, ABM2022, ABMV2014, ABSVsinum16, AMSLP2018, BBDMR2017, BBM,CGPS,CG2017, FS-M2AN18, GMV2018, GV:IMA2017,  MR4253143, MR4229296, MRR2015}  and the references therein.

Despite to the fact of the important contributions of VEM in different subjects, there is a new approach to this method which for the best of the authors knowledge is available for  second order elliptic problems, where the standard hypotheses of \cite{BBCMMR2013} can be relaxed, and its related to the size of the edges of the elements on the polygonal meshes. We know so far that the classic VEM requires that the elements must be star shaped and with sufficiently large edges. It is precisely this last assumption that is relaxed in the works of \cite{BLR2017,MR3815658} where arbitrary edges or feces, depending  on the dimension in which the problem is stated,  are now allowed. These references show that for more general assumptions, and suitable stabilizations, it is possible to 
obtain stability of the VEM and error estimates. This is an ongoing subject of research, and the available results are for VEM spaces to discretize $H^1$. Also, recently on \cite{ALR:22,MR4359996,MR4284360,MR4461634}  is possible to find applications of the small edges approach.

It is important to take into account that the theoretical analysis for a VEM allowing small edges needs to pay a price, in the sense that not any system of partial differential equations allows to consider this approach. More precisely, the regularity of the solution plays a role for the analysis. In \cite{BLR2017,MR3815658} the authors have shown that when the solution of the PDE is such that $H^{1+s}$ with $s>3/2$, the approximation properties for the VEM with small edges hold. This implies that, under some geometrical hypotheses on the domain, boundary conditions, data, or physical parameters, this new nature of the  VEM is possible to be applied. Hence, for our purposes and for the best of our knowledge, we cannot go far from this regularity requirement.

In the present paper, we continue with our research program of VEM with small edges for second order elliptic problems. More precisely,  our  contribution is to apply this approach on two problems: in one hand, we have the convection-difussion-reaction problem and, on the other, the convection-difussion eigenvalue problem.  These two problems are of importance due to the applicability of such equations. For the convection-difussion-reaction problem, we focus on the load problem since in the eigenvalue problem the term associated to the reaction is similar to the right hand side that has the eigenvalue of the problem. This is the reason why we only consider the difussion-convection problem on the spectral setting. Also, since our intention is to apply the small edges scheme for the VEM, as we have claimed before, we need to operate under a suitable geometrical setting which in our case, consists in a open, bounded, and convex two dimensional domain with null Dirichlet boundary conditions. These assumptions will be required for both, the load and spectral problems.

%The classic variational formulation of this non-symmetric problem is $H^1$ type and fits in the available theory of \cite{BLR2017}. With the regularity
%of the solution, we prove error estimates under seminorms depending on a suitable stabilization term.

The outline of our manuscript is the following: In section \ref{sec:model} we present the source model problem, the bilinear forms which consider, the well posedness and regularity 
properties. Section \ref{sec:virtual} is the core of our paper, where the virtual element methods are introduced, under the assumption of small edges for  the  polygonal meshes. Here
we define the local and global virtual spaces and the discrete bilinear forms. With these ingredients at hand, we discretize the source problem introduced in section \ref{sec:model}, proving the well posedeness of the discrete problem and, under the weaker assumptions of the mesh, we derive error estimates. As an application of the derived results so far, in section \ref{sec:spectral} we analyze the eigenvalue problem associated to \eqref{eq:state_equation}. Since the spectral problem is nonsymmetric, the analysis is performed introducing
the adjoint eigenvalue problem. Spurious free, convergence and error estimates results are proved for our proposed VEM. Finally, in section \ref{sec:numerics} we illustrate the theoretical results of the small edges approach  for the source and eigenvalue problem, reporting a set of numerical tests in different contexts and geometries.

\section{The variational problem}
\label{sec:model}

The weak formulation of \eqref{eq:state_equation} reads as follows: Find $u \in H_0^1(\Omega)$ such that
\begin{equation}
\label{eq:weak_state_equation}
\CB(u,v)=\CF(v) \quad \forall v \in H_{0}^{1}(\Omega),
\end{equation}
where $\CB:H_0^1(\Omega)\times H_0^1(\Omega)\rightarrow\mathbb{R}$ is the bilinear form  defined by 
\begin{equation}
\label{eq:bilinear_cont}
\CB(w,v):=a(w,v)+b(w,v)+c(w,v),\quad \forall w,v\in H_{0}^{1}(\Omega),
\end{equation}
and $\CF:H_0^1(\Omega)\rightarrow\mathbb{R}$  is the functional defined by 

\begin{equation}
\label{eq:functionaF}
	\CF(v)=\int_{\Omega}fv,\quad \forall v\in L^2(\Omega),
\end{equation}
\noindent respectively, with $a(\cdot,\cdot)$, $b(\cdot,\cdot)$,  and $c(\cdot,\cdot)$ being bounded bilinear forms  defined as follows
\begin{multline}
	a:H^1(\Omega)\times H^1(\Omega) \to \R; \qquad 
	a(w,v):=\int_{\Omega}\kappa(\mathbf{x}) \nabla w\cdot\nabla v,\forall w,v\in H_0^1(\Omega);
	\label{def_a}\\
	b:H^1(\Omega)\times L^2(\Omega) \to \R; \qquad 
	b(w,v):=\int_{\Omega}( \mathbf{\vartheta}(\mathbf{x})\cdot \nabla w)v,\forall w\in H_0^1(\Omega),v\in L^2(\Omega);\\
	c:L^2(\Omega)\times L^2(\Omega) \to \R; \qquad 
	c(w,v):=\int_{\Omega}\gamma(\mathbf{x}) wv,\forall w,v\in  L^2(\Omega).
\end{multline}

%It is important to remark that \eqref{eq:weak_state_equation} is not a symmetric problem due to the presence of  the bilinear form $b(\cdot,\cdot)$ and hence, the numerical method the we propose below involves a non-symmetric matrix as well.

The assumptions on the coefficients on \eqref{eq:state_equation} lead us to the continuity of $\mathcal{B}(\cdot,\cdot)$, i.e, there exists a constant $M_1>0$ such that
\begin{equation}
\label{eq:acotamientoB}
\mathcal{B}(w,v)\leq M_1\|w\|_{1,\O}\|v\|_{1,\O}\quad\forall v,w\in H_0^1(\O).
\end{equation}

Also, the following condition holds
\begin{equation}
\label{eq:inf-supB}
\displaystyle\sup_{v\in H_0^1(\Omega)}\frac{\mathcal{B}(w,v)}{\|v\|_{1,\O}}\geq \beta\|w\|_{1,\O}\quad \forall w\in H_0^1(\O),
\end{equation}
where $\beta>0$ is a constant independent of $v$. Hence, \eqref{eq:acotamientoB} and \eqref{eq:inf-supB} implies the well posedness
of \eqref{eq:weak_state_equation}.

On the other hand, for the implementation of the virtual element method of our interest, the regularity of  \eqref{eq:weak_state_equation} is a key ingredient in order to
obtain approximation properties. Given $f\in H^{-1}$, there exist  positive constants $C^{*},C^{**}$ such that the solution $u$ of \eqref{eq:weak_state_equation}
satisfies
\begin{equation}
\label{eq:regularity}
\|u\|_{1,\Omega}\leq C^{*}\|f\|_{-1,\Omega}\quad\text{and}\quad
\|u\|_{2,\Omega}\leq C^{**}\|f\|_{0,\Omega}.
\end{equation}

\section{The virtual element method}
\label{sec:virtual}

In this section we briefly review a virtual element method (VEM) for te system \eqref{eq:weak_state_equation}. First we recall the mesh construction and the assumptions considered in \cite{BBCMMR2013} for the virtual element
method. 
Let $\left\{\CT_h\right\}_h$ be a sequence of decompositions of $\Omega$ into polygons, $E$. Let $h_E$ denote the diameter of the element $E$ and $h$ the maximum of the diameters of all the elements of the mesh, i.e., $h:=\max_{E\in\Omega}h_E$.  Moreover, for simplicity in what follows we assume that $\kappa $ and $\gamma$  are  piecewise constant with  respect to the decomposition $\mathcal{T}_h$, i.e., they are  piecewise constants for all $E\in \mathcal{T}_h$ (see for instance \cite{BLR2017}).

 For the analysis of the VEM, we will make as in \cite{BBCMMR2013} the following
assumption: there exists a positive real number $\rho$ such that, for every $ E \in\CT_{h}$ and for every $\CT_{h}$,
\begin{itemize}
\item \textbf{A1.} For all meshes
$\CT_h$, each polygon $E\in\CT_h$ is star-shaped with respect to a ball
of radius greater than or equal to $\rho h_{E}$.
\end{itemize}

For any simple polygon $ E$ we define 
\begin{align*}
	\widetilde{V}_h^E:=\{v_h\in H^{1}(E):\Delta v_h \in \mathbb{P}_1(E),
	v_h|_{\partial E}\in C^0(\partial E), v_h|_{e}\in \mathbb{P}_1(e) \  \forall e \in \partial E  \}.
\end{align*}

Now, in order to choose the degrees of freedom for $\widetilde{V}_h^E$ we define
\begin{itemize}
	\item $\mathcal{V}_E^h$: the value of $w_{h}$ at each vertex of $E$,
\end{itemize}
as a set of linear operators from $\widetilde{V}_h^E$ into $\R$. In \cite{AABMR13} it was established that $\mathcal{V}_E^h$ constitutes a set of degrees of freedom for the space $\widetilde{V}_h^E$.

 On the other hand, we define  the projector $\PiK:\ \WK\longrightarrow\P_1(E)\subseteq\WK$ for
each $v_{h}\in\WK$ as the solution of 
\begin{subequations}
\begin{align*}
\int_E (\nabla\PiK v_{h}-\nabla v_{h})\cdot\nabla q&=0
\qquad\forall q\in\P_1(E),
%\label{eq:proj1}
\\
\overline{\PiK v_{h}}&=\overline{v_{h}},
%\label{eq:proj2}
\end{align*}
\end{subequations}
where  for any sufficiently regular
function $v$, we set 
\begin{equation*}
\label{eq:fixed}
\overline{v}:=\vert\partial E\vert^{-1}\int_{\partial E}v.
\end{equation*}
We observe that the term $\PiK v_{h}$ is well defined and computable from the degrees of freedom  of $v$ given by $\mathcal{V}_E^h$, and in addition the projector $\PiK$ satisfies the identity  $\PiK(\P_{1}(E))=\P_{1}(E)$ (see for instance \cite{AABMR13}).

We are now in position  to introduce our local virtual space
\begin{equation}\label{Vk}
\VK
:=\left\{v_{h}\in 
\WK: \displaystyle \int_E \PiK v_{h}p=\displaystyle \int_E v_{h}p,\quad \forall p\in 
\mathbb{P}_1(E)\right\}.
\end{equation}
Now, since $\VK\subset \WK$, the operator $\PiK$ is well defined on $\VK$ and computable  only on the basis of the output values of the operators in $\mathcal{V}_E^h$.
In addition, due to the particular property appearing in definition of the space $\VK$, it can be seen that $\forall p \in \mathbb{P}_1(E)$ and $\forall v_h\in \VK$ the term $(v_h,p)_{0,E}$
is computable from $\PiK v_h$, and hence  the  ${\mathrm L}^2(E)$-projector operator $\Pio: \VK\to \P_1(E)$ defined  by
$$\int_E \Pi^{E}v_h=\int_E v_h p\qquad \forall p\in \mathbb{P}_1(E),$$
depends only on the values of the degrees of freedom of $v_h$.  Actually, it is easy to check that the projectors $\PiK$ and $\Pi^{E}$ are the same operators  on the space $\VK$ (see \cite{AABMR13} for further details).

Finally, for every decomposition $\CT_h$ of $\Omega$ into simple polygons $ E$ we define the global virtual space
\begin{equation}
\label{eq:globa_space}
\Vh:=\left\{v\in H^{1}(E):\ v|_{ E}\in\VK\quad\forall E\in\CT_h\right\},
\end{equation}
and the global degrees of freedom are obtained by collecting the local ones, with the nodal and interface degrees of freedom corresponding to internal entities counted only once
those on the boundary are fixed to be equal to zero in accordance with the ambient space $H_{0}^{1}(\Omega)$.

\subsection{Discrete formulation}
 In order to construct the discrete scheme, we need some preliminary definitions. First, we split the bilinear form $\CB(\cdot, \cdot)$ as follows:
 \begin{equation*}
\label{eq:bilineal_form_B_split}
\CB(w,v):=\sum_{ E\in\CT_h} \CB^{E}(w,v):=\sum_{ E\in\CT_h} a^{E}(w,v)+ b^{E}( w,v)+c^{E}( w,v)\,\, \forall w,v\in H^{1}(\Omega),
\end{equation*}
\noindent where
\begin{equation*}
	 a^{E}(w,v):=\int_E \kappa(x) \nabla w\cdot\nabla v, \quad
	  b^{E}(w,v):=\int_E(\mathbf{\vartheta}(x)\cdot \nabla w)v,\quad
	 c^{E}(w,v):=\int_E\gamma(x) w v.
	%\label{def_cE}
\end{equation*}
%\begin{align*}
%	 &a^{E}(w,v):=\int_E \kappa(x) \nabla w\cdot\nabla v, \\
%	 & b^{E}(w,v):=\int_E(\mathbf{\vartheta}(x)\cdot \nabla w)v,\\
%	 &c^{E}(w,v):=\int_E\gamma(x) w v.
%	%\label{def_cE}
%\end{align*}

Now, in order to propose the discrete bilinear form for $a(\cdot,\cdot)$ (cf. \eqref{def_a}), we consider the following symmetric and semi-positive definite bilinear form  $S^{E}:\VK\times\VK\to \mathbb{R}$ introduced in \cite{BLR2017}. For each $E\in \mathcal{T}_h$ and for all $w_h,v_h\in \VK$ we set
\begin{equation*}
\label{20}
S^{E}(w_h,v_h):=S^{\partial}(w_h,v_h), %h_{E}\int_{\partial{E}}\partial_{s}w_h\partial_s
%v_h\qquad\forall w_h,v_h\in\VK, 
\end{equation*}
where, if $\partial_s$ denotes a derivative along the edge, $S^{\partial }(\cdot,\cdot)$ is defined by (see \cite{WRR2016})
%\begin{align*}
%&S^{\circ}(w_h,v_h):=\sum\limits_{i=1}^{N_E}w_h(P_i)v_h(P_i)
% \quad \forall v_h,v_h\in \VK,\nonumber
%\end{align*}
%and
\begin{equation}
\label{eq:derivative_stab}
S^{\partial }(w_h,v_h):=h_{E}\int_{\partial{E}}\partial_{s}w_h\partial_s
v_h\qquad\forall w_h,v_h\in\VK.
\end{equation}

Then, we introduce on each element $E$ the local (and computable) bilinear forms 
%\EO{	
\begin{itemize}
\item $a_h^E(w_h, v_h)  :=a^{E}(\PiK w_h,\PiK v_h) +S^E( w_h-\PiK w_h, v_h-\PiK  v_h)$,
%\item $a_h^E(w_h, v_h)  :=\displaystyle\int_E \kappa(x)\nabla( \Pi_{E}^0 w_h)\cdot \nabla(\Pi_{E}^0 v_h)+S^E( w_h-\PiK w_h, v_h-\PiK  v_h)$,
\vspace{0.4cm}
\item $b_h^E( w_h, v_h)  :=b^E(\PiK w_h,\Pi^{E}v_h)$,
%\item $b_h^E( w_h, v_h)  :=\displaystyle\int_E (\mathbf{\vartheta}(x) \cdot \nabla(\Pi_{k-1}^0 w_{h})){\Pi}_{k-1}^0 v_{h}$, 
\vspace{0.4cm}
\item $c_h^E( w_h, v_h)  := c^E(\Pi^{E}w_h,\Pi^{E}v_h)$,
%\item $c_h^E( w_h, v_h)  :=\displaystyle\int_E\gamma(x) \Pi_{k-1}^0 w_{h}\Pi_{k-1}^0 v_{h}$,
\end{itemize}
%}
for all $w_h,v_h\in \VK.$

Now we introduce the following discrete semi-norm:
\begin{equation}
\label{eq:triple}
\vertiii{v}_{E}^2:=a^{E}\big(\PiK v,\PiK v)+S^{E}(v-\bar{v},v-\bar{v})\qquad\forall
v\in\VK+\mathcal{V}^{E},
\end{equation}
where  $\mathcal{V}^{E}\subseteq H^1(E)$ is a subspace of sufficiently
regular functions for $S^{E}(\cdot,\cdot)$ to make sense.
 
For any sufficiently regular functions,
we introduce the following global semi-norms
$$
\vertiii{v}^2:=\sum_{E\in\CT_h}\vertiii{v}_{E}^2,\qquad
\left|v\right|_{1,h}^2
:=\sum_{E\in\CT_h}\left\|\nabla v\right\|_{0,E}^2.
$$
It has been proved in  \cite[Lemma~3.1]{BLR2017} the existence of  positive constants $C_1,C_2, C_3$, independent of $h$, but depending on $\kappa$
such that
\begin{align}
C_1\vertiii{v}_{E}^2\le a_h^{E}(v,v)\le C_2\vertiii{v}_{E}^2\quad\forall v\in\VK,\label{eqrefgt1}\\
a_h^{E}(v,v)\le C_3(\vertiii{v}^2+\vert v\vert_{1,E}^2)\quad\forall v\in\VK.\label{eqrefgt12}
\end{align}
In addition, it holds
\begin{align}
a^{E}(v,v)\le C_4\vertiii{v}_{E}^2\quad\forall v\in\VK,\label{eqrefgt2}\\
\vertiii{p}_{E}^2\le C_5 a^{E}(p,p)\quad\forall p\in\P_1(E),\label{eqrefgt3}
\end{align}
where $C_4,C_5$ are positive constants independent of $h$.

As is customary, the bilinear form $\mathcal{B}_h(\cdot,\cdot)$ can be expressed componentwise as follows
 \begin{equation}
\label{eq:bilineal_form_B_split_{h}}
\CB_{h}(w_{h},v_{h}):= \sum_{ E\in\CT_h}\CB_{h}^{E}(w_{h},v_{h})
= \sum_{ E\in\CT_h} a_{h}^{E}( w_{h}, v_{h})+ b_{h}^{E}( w_{h},v_{h})+ c_{h}^{E}( w_{h},v_{h}).
\end{equation}

Now, we are in a position to write the virtual element discretization   for problem \eqref{eq:weak_state_equation}: Find $u_{h} \in \Vh$ such that
\begin{equation}
\label{eq:virtual}
\CB_{h}(u_{h},v_{h})=\mathcal{F}(v_{h}) \quad \forall v_{h} \in \Vh,
\end{equation}
where $\mathcal{F}(v_{h})=\displaystyle(f_h,v_h)_{0,\O}$ with $f_{h}:=\Pi^{E} f$ and $(\cdot,\cdot)_{0,\O}$ denotes the inner product in $L^2(\O)$. It is clear that  $\mathcal{B}_h(\cdot,\cdot)$ is continuous, i.e, 
\begin{equation}\label{eq:cont_disc}
\mathcal{B}_h(w_h,v_h)\leq M_2\vertiii{w_h}\vertiii{v_h}\quad v_h,w_h\in V_h^E,
\end{equation}
with $M_2$ being a  constant independent of $h$.

From  the definition of the bilinear form $a_h(\cdot,\cdot)$, it is easy to check that this discrete bilinear form is coercive in $V_h$. Indeed, let $v_h\in V_h$. Then, 
\begin{multline}
\label{eq.elliptich}
a_{h}(v_h,v_h)
=\sum_{E\in\CT_h}a_h^{E}\big(v_h,v_h\big)
\geq\sum_{E\in\CT_h}C_1\vertiii{v_{h}}_{E}^2
\geq C\sum_{E\in\CT_h}a^{E}\big(v_h,v_h\big)\\
\geq C\vert v_h\vert_{1,\O}^2
\geq \beta\left\|v_h\right\|_{1,\O}^2
\qquad\forall v_h\in\Vh,
\end{multline}
where we have used \eqref{eqrefgt1}, \eqref{eqrefgt2} and
the generalized Poincar\'e inequality.

On the other hand, we also have the following approximation result for polynomials in star-shaped domain (see for instance \cite{BS-2008}).
\begin{lemma}
\label{lmm:bh}
If the assumption {\bf A1} is satisfied, then there exists a constant
$C$, depending only on $k$ and $\g$, such that for every $s$ with 
$0\le s\le k$ and for every $v\in\HusE$, there exists
$v_{\pi}\in\P_k(E)$ such that
$$
\left\|v-v_{\pi}\right\|_{0,E}
+h_{E}\left|v-v_{\pi}\right|_{1,E}
\le Ch_{E}^{1+s}\left\|v\right\|_{1+s,E}.
$$
\end{lemma}
The next step is to find appropriate terms $u_I$ and $u_{\pi}$ that can
be used in the above lemma to prove the claimed convergence. For the
latter we have the following proposition, which is derived by
interpolation between Sobolev spaces (see for instance
\cite[Theorem~I.1.4]{MR851383}) from the analogous result for integer values
of $s$. In its turn, the result for integer values is stated in
\cite[Proposition~4.2]{BBCMMR2013} and follows from the classical
Scott-Dupont theory (see \cite{BS-2008}).

The following, proved in \cite[Lemma 4.3]{MR4284360},
is an extension of \cite[Proposition~4.3]{BBCMMR2013} to
less regular functions.

\begin{lemma}
\label{estima2}
Under the assumption  {\bf A1}, for
each $s$ with $0<s\le k$, there exist $\widehat{\sigma}$ and a constant $C$, depending only
on $k$, such that for every $v\in H^{1+s}(\O)$, there exists
$v_I\in\Vh$ that satisfies
\begin{align*}
%\label{firststimate}
\left|v-v_I\right|_{1+t,E}&\le Ch_{E}^{s-t}\left|v\right|_{1+s,E}\qquad 0\leq t\leq\min\{\widehat{\sigma},s\},\\ 
%\label{secondstimate}
\left\|v-v_I\right\|_{0,E}
&\le Ch_{E}\left|v\right|_{1+s,E}.
\end{align*}
\end{lemma}
In order to prove that the virtual element discretization \eqref{eq:virtual} is well defined, we need the following technical results.
\begin{lemma}
\label{lmm:technical1}
For every $w^*\in H_{0}^{1}(\O)$ there exists  $w^*_{h}\in V_h$ such that 
\begin{equation}\label{eq:auxiliar}
a_{h}(w_{h}^*,v_{h})=a(w^*,v_{h})\qquad \forall v_{h}\in V_h.
\end{equation}
Moreover, there exists a constant $C_\kappa$, independent of $h$, such that
\begin{equation}\label{eq:auxiliarII}
h\|w^*-w_{h}^*\|_{1,\O}+\|w^*-w_{h}^*\|_{0,\O}\leq C_\kappa h\|w^*\|_{1,\O}.
\end{equation}
\end{lemma}
\begin{proof}
From \eqref{eq.elliptich}, a simple application of Lax-Milgram lemma implies that \eqref{eq:auxiliar} has a unique solution and $\vertiii{w_{h}^*}\leq C\|w^*\|_{1,\O}$. 
On the other hand, to prove \eqref{eq:auxiliarII}, we note that:
\begin{align}\label{eq:intermediaI}
\nonumber 
|w^*-w_{h}^*|_{1,\O}^{2}&\leq C_\kappa(a(w^*,w^*-w_{h}^*)-a(w_{h}^*,w^*-w_{h}^*))\\\nonumber 
&\leq C_\kappa\left(|w^*|_{1,\O}+|w^*_h|_{1,\O}\right)|w^*-w_{h}^*|_{1,\O}\\
&\leq \C_\kappa\left(|w^*|_{1,\O}+\vertiii{w_{h}^*}\right)|w^*-w_{h}^*|_{1,\O}\leq C_\kappa|w^*|_{1,\O}|w^*-w_{h}^*|_{1,\O},
\end{align}
where we have used \eqref{eqrefgt2}. Now to control the error in $L^2$ norm we use a standard  duality argument.  Let $\phi\in H^{2}(\O)\cap H_{0}^{1}(\O)$ be the solution of
\begin{equation}\label{eq:duality}
a(w,\phi)=(w^*-w_{h}^*,w)_{0,\O}\qquad w\in H_{0}^{1}(\O),
\end{equation}
and let $\phi_{I}\in V_h^E$ be its interpolant.  Then, from Lemma \ref{estima2} we have
\begin{equation}\label{eq:aprox_dual}
|\phi-\phi_{I}|_{1,\O}\leq Ch|\phi|_{2}\leq Ch\|w^*-w_{h}^*\|_{0,\O}.
\end{equation} 
Then, testing  \eqref{eq:duality} with $w=w^*-w_{h}^*\in H_{0}^{1}(\O)$, and using \eqref{eq:auxiliar} we obtain
\begin{multline}\label{eq:intermedia}
\|w^*-w_{h}^*\|_{0,\O}^{2}=a(w^*-w_{h}^*,\phi)=a(w^*-w_{h}^*,\phi-\phi_{I})+a(w^*-w_{h}^*,\phi_{I})\\
=a(w^*-w_{h}^*,\phi-\phi_{I})+a_{h}(w_{h}^*,\phi_{I})-a(w_{h}^*,\phi_{I})\\
=a(w^*-w_{h}^*,\phi-\phi_{I})+\sum_{E\in\CT_{h}}\left(a_{h}^{E}(w_{h}^*,\phi_{I}-\PiK\phi_{I})-a^{E}(w_{h}^*,\phi_{I}-\PiK\phi_{I})\right)\\
\leq C_\kappa\left(|w^*-w_{h}^*|_{1,\O}|\phi-\phi_{I}|_{1,\O}+\sum_{E\in\CT_{h}}\vertiii{w_{h}^*}_{E}\vertiii{\phi_{I}-\PiK\phi_{I}}_{E}\right.\\
\left.+\sum_{E\in\CT_{h}}|w_{h}^*|_{1,E}|\phi_{I}-\PiK\phi_{I}|_{1,E}\right)\\
\leq C_\kappa\left(h\|w^*\|_{1,\O}|\phi|_{2,\O}+\sum_{E\in\CT_{h}}\vertiii{w_{h}^*}_{E}\vertiii{\phi_{I}-\PiK\phi_{I}}_{E}\right),
\end{multline}
where, for the last estimate, we have used \eqref{eq:intermediaI}, \eqref{eq:aprox_dual} and \eqref{eqrefgt2}.
The following step is to bound the last term of the above estimate. With this purpose, we first note that from \eqref{eqrefgt1}, 
 the definition of $a_h^E(\cdot,\cdot)$ and operating as in the proof of \cite[Theorem 4.5]{BLR2017}, together with the fact that $\Pi^{\nabla,E}\phi_I$ is the best approximation of 
 $\phi_I$ through polynomials of degree $k$,  we obtain
\begin{align*}
C_{1}&\vertiii{\phi_{I}-\PiK\phi_{I}}_{E}^{2}\leq a_{h}^{E}(\phi_{I}-\PiK\phi_{I},\phi_{I}-\PiK\phi_{I})\\
&=S^{E}(\phi_{I}-\PiK\phi_{I},\phi_{I}-\PiK\phi_{I})\leq h_{E}|\phi_{I}-\PiK\phi_{I}|_{1,\partial E}^2\\
%&\GR{\leq C |\phi_{I}-\PiK\phi_{I}|_{1,E}^2+h^2|\phi_{I}-\PiK\phi_{I}|_{2,E}^2.}\\
&\leq C_\kappa (|\phi_{I}-\PiK\phi_{I}|_{1,E}^2+h^2|\phi_{I}|_{2,E}^2.)\\
&\leq C_\kappa(|\phi_{I}-\PiK\phi|_{1,E}^2+h^2|\phi_{I}-\phi|_{2,E}^2+h^2|\phi|_{2,E}^2)\leq Ch^2|\phi|_{2,E}^2,
%&\leq |\PiK\phi_I-\phi_I|_{1,E}^2+|\phi_I-\Pi^E\phi_I|_{1,E}^2+h_{E}|\phi_{I}-\PiK\phi_{I}|_{1,\partial E}^2\\
%&\leq C(2|\phi_I-\Pi^E\phi_I|_{1,E}^2)+h_{E}|\phi_{I}-\Pi^{E}\phi_{I}|_{1,\partial E}^2\leq C\left(|\phi_{I}-\Pi^{E}\phi_{I}|_{1,E}^{2}+h_{E}^{2}|\phi_{I}|_{2,E}\right),
\end{align*}
allowing us to conclude that 
\begin{align*}
\sum_{E\in\CT_{h}}\vertiii{w_{h}^*}_{E}\vertiii{\phi_{I}-\PiK\phi_{I}}_{E}\leq C h\vertiii{w_{h}^*}|\phi|_{2,\O}.
\end{align*}
Therefore, from  the above estimation, together with \eqref{eqrefgt2},  \eqref{eq:intermedia} and  the fact that $\vertiii{w_{h}^*}\leq C\|w^*\|_{1,\O}$, we obtain 
\begin{align*}
\|w^*-w_{h}^*\|_{0,\O}^{2}\leq C_\kappa h\|w^*\|_{1,\O}\|w^*-w_{h}^*\|_{0,\O}.
\end{align*}
This concludes the proof.
\end{proof}

In order to state the well posedness of the discrete problem \eqref{eq:virtual},  the following discrete inf-sup condition is essential.
\begin{lemma}
\label{lmm:disc_inf_sup}
There exists a constant $\widehat{\beta}>0$ such that, for all $h<h_0$:
\begin{equation}\label{eq:inf-supBh}
\displaystyle\sup_{w_h\in V_h}\frac{\mathcal{B}_h(v_h,w_h)}{\|w_h\|_{1,\O}}\geq\widehat{\beta}\|v_h\|_{1,\Omega}\quad\forall v_h\in V_h.
\end{equation}
%\begin{equation}\label{eq:inf-supBh}
%\displaystyle\sup_{w_h\in \V_h}\frac{\mathcal{B}_h(v_h,w_h)}{\GR{\vertiii{w_h}}}\geq\widehat{\beta}\|v_h\|_{1,\Omega}\quad\forall v_h\in V_h.
%\end{equation}
\end{lemma}
\begin{proof}
To prove this result we resort to the classic construction of the Fortin operator. From the continuous inf-sup condition 
\eqref{eq:inf-supB}, for $v_h\in V_h$, there exists $w^*\in H_0^1(\Omega)$ such that 
\begin{equation}\label{eq:inf-supB2}
\frac{\mathcal{B}(v_h,w^*)}{\|w^*\|_{1,\O}}\geq M_2\|v_h\|_{1,\O}.
\end{equation}
From Lemma \ref{lmm:technical1} the following problem
$$\text{Find }w_h^*\in \Vh \text{ such that } a_h(w_h^*,v_h)=a(w^*,v_h),$$
has a unique solution which satisfies
\begin{equation}
\label{eq:aux_forinf-sup}\vertiii{w_h^*}\leq \|w^*\|_{1,\O},\qquad \|w_h^*-w^*\|_{0,\O}\leq h\|w^*\|_{1,\O}.
\end{equation}

On the other hand,  elementary algebraic manipulations, together with the use of the properties of the virtual projector,
lead to
\begin{align*}
\mathcal{B}_h(v_h,w_h^*)&=a_h(v_h,w_h^*)+b_h(v_h,w_h^*)+c_h(v_h,w_h^*)\\
&=a(v_h,w^*)+b_h(v_h,w_h^*)-b(v_h,w^*)+c_h(v_h,w_h^*)-c(v_h,w^*)\\
&+b(v_h,w^*)+c(v_h,w^*)\\
&=\mathcal{B}(v_h,w^*)+b_h(v_h,w_h^*)-b(v_h,w^*)+c_h(v_h,w_h^*)-c(v_h,w^*)\\
&=\mathcal{B}(v_h,w^*)+b_h(v_h,w_h^*)-b(v_h,w_h^*)+c_h(v_h,w_h^*)-c(v_h,w_h^*)\\
&+b(v_h,w_h^*-w^*)+c(v_h,w_h^*-w^*).
\end{align*}
Reorganizing the computations above, in order to simplify the presentation of the matetrial, we have
\begin{align}\nonumber
\mathcal{B}_h(v_h,w_h^*)+\underbrace{b(v_h,w^*-w_h^*)}_{T_1}&+\underbrace{c(v_h,w^*-w_h^*)}_{T_2}\\\label{eq:estimacioninf-sup}
+\underbrace{c(v_h,w_h^*)-c_h(v_h,w_h^*)}_{T_3}&+\underbrace{b(v_h,w_h^*)-b_h(v_h,w_h^*)}_{T_4}=\mathcal{B}(v_h,w^*).
\end{align}

 Now our task is to estimate each of the terms  $T_i$, $i\in\{1, 2, 3, 4\}$, independently of the meshsize $h$. We begin with $T_1$. Observe that  using that fact that  $b(\cdot,\cdot)$ is bounded, together with  \eqref{eq:aux_forinf-sup}, we obtain
 \begin{equation}\label{eq:primeracota}
T_1\leq C_{\mathbf{\vartheta}}h\|v_h\|_{1,\O}\|w^*\|_{1,\O}.
\end{equation}
For the term $T_2$, a simple application of the  Cauchy-Schwarz inequality, together with  \eqref{eq:aux_forinf-sup}, leads to 
\begin{equation*}\label{eq:segundacota}
T_2\leq C_{\gamma}h\|v_h\|_{1,\O}\|w^*\|_{1,\O}.
\end{equation*}
For  $T_3$, using that  $\gamma(\mathbf{x}) $ is piecewise constant with respect to the meshes, together with the properties of projector $\Pi^{E}$ and \eqref{eq:aux_forinf-sup}, we have
\begin{multline*}
%\nonumber
T_3=\sum_{E\in\CT_h}c^E(v_h,w_h^*)-c_h^E(v_h,w_h^*)
=\sum_{E\in\CT_h}c^E(v_h,w_h^*-\Pi^{E}w_h^*)-c_h^E(v_h,w_h^*-\Pi^{E}w_h^*)\\\nonumber
\leq\sum_{E\in\CT_h}c_{\gamma}\|v_h\|_{0,E}\|w_h^*-\Pi^{E}w_h^*\|_{0,E}\leq\sum_{E\in\CT_h}c_{\gamma}h_E\|v_h\|_{0,E}|w_h^*|_{1,E}\\\label{eq:terceracota}
\leq \sum_{E\in\CT_h}c_{\gamma,\kappa}h_EC_4\|v_h\|_{0,E}\vertiii{w_h^*}_{E}\leq C_{\gamma,\kappa}h\|v_h\|_{1,\O}\vertiii{w_h^*}\leq C_{\gamma,\kappa}h\|v_h\|_{1,\O}\|w^*\|_{1,\O},
\end{multline*} 
where we have used   that $\kappa(\mathbf{x}) $ is piecewise constant with respect to the meshes, together with \eqref{eqrefgt2} and \eqref{eq:aux_forinf-sup}.
For the term $T_4$, first, it is necessary to note the following relation
\begin{align*}
b^E(v_h,w_h^*)&-b_h^E(v_h,w_h^*)=\left(\mathbf{\vartheta}(\mathbf{x})\cdot\nabla v_h,w_h^*\right)_{0,E}-\left(\mathbf{\vartheta}(\mathbf{x})\cdot\nabla \PiK v_h,\Pi^{E}w_h^*\right)_{0,E}\\
&=\left(\mathbf{\vartheta}(\mathbf{x})\cdot\nabla v_h,w_h^*-\Pi^Ew_h^*\right)_{0,E}+\left(\mathbf{\vartheta}(\mathbf{x})\cdot(\nabla v_h-\nabla \PiK v_h),\Pi^{E}w_h^*\right)_{0,E}\\
&=\left(\mathbf{\vartheta}(\mathbf{x})\cdot\nabla v_h-\Pi^E(\mathbf{\vartheta}(\mathbf{x})\cdot\nabla v_h),w_h^*-\Pi^Ew_h^*\right)_{0,E}\\
&+\left(\nabla v_h -\nabla \PiK v_h,\mathbf{\vartheta}(\mathbf{x})\Pi^{E}w_h^*-\mathbf{\vartheta}(\mathbf{x})w_h^*\right)_{0,E}\\
&+\left(\nabla v_h -\nabla \PiK v_h,\mathbf{\vartheta}(\mathbf{x})w_h^*-\Pi^E(\mathbf{\vartheta}(\mathbf{x})w_h^*)\right)_{0,E},
\end{align*}
where we have used the properties of the virtual projector. Therefore, we have the following estimate
%\begin{align*}
%b^E(v_h,w_h^*)&-b_h^E(v_h,w_h^*)=\left(\mathbf{\vartheta}\cdot\nabla v_h-\Pi_1^E(\mathbf{\vartheta}\cdot\nabla v_h),w_h^*-\Pi_{1}^Ew_h^*\right)_{0,E}\\
%&\hspace{2.5 cm}+\left(\nabla v_h-\nabla \PiK v_h,\mathbf{\vartheta}w_h^*-\Pi_{0}^{E}\left(\mathbf{\vartheta}w_h^*\right)\right)_{0,E}\\
%&\hspace{2.5 cm}+\left(\nabla v_h-\nabla \PiK v_h,\mathbf{\vartheta}\left(\Pi^{E}w_h^*-w_h^*\right)\right)_{0,E}\\
%&\leq\|\mathbf{\vartheta}\cdot\nabla v_h-\Pi_1^E(\mathbf{\vartheta}\cdot\nabla v_h)\|_{0,E}\|w_h^*-\Pi_{1}^Ew_h^*\|_{0,E}\\
%&+|v_h-\PiK v_h|_{1,E}\left(\|\mathbf{\vartheta}w_h^*-\Pi_{0}^{E}\left(\mathbf{\vartheta}w_h^*\right)\|_{0,E}+\|\mathbf{\vartheta}\left(\Pi^{E}w_h^*-w_h^*\right)\|_{0,E}\right)\\ 
%&\leq h_EC_{\mathbf{\vartheta}}\|v_h\|_{1,E}|w_h^*|_{1,E}\leq h_EC_{\kappa,\mathbf{\vartheta}}C_4\|v_h\|_{1,E}\vertiii{w_h^*},
%\end{align*} 
\begin{multline*}
b^E(v_h,w_h^*)-b_h^E(v_h,w_h^*)\leq\|\mathbf{\vartheta}(\mathbf{x})\cdot\nabla v_h-\Pi_1^E(\mathbf{\vartheta}(\mathbf{x})\cdot\nabla v_h)\|_{0,E}\|w_h^*-\Pi^Ew_h^*\|_{0,E}\\
+\|\nabla v_h -\nabla \PiK v_h\|_{0,E}\|\mathbf{\vartheta}(\mathbf{x})\Pi^{E}w_h^*-\mathbf{\vartheta}(\mathbf{x})w_h^*\|_{0,E}\\
+\|\nabla v_h -\nabla \PiK v_h\|_{0,E}\|\mathbf{\vartheta}(\mathbf{x})w_h^*-\Pi^E(\mathbf{\vartheta}(\mathbf{x})w_h^*)\|_{0,E}\\%+\left(\mathbf{\vartheta}(\mathbf{x})\cdot(\nabla (v_h -\PiK v_h),\Pi^{E}w_h^*\right)_{0,E}\\
%&+|v_h-\PiK v_h|_{1,E}\left(\|\mathbf{\vartheta}w_h^*-\Pi_{0}^{E}\left(\mathbf{\vartheta}w_h^*\right)\|_{0,E}+\|\mathbf{\vartheta}\left(\Pi^{E}w_h^*-w_h^*\right)\|_{0,E}\right)\\ 
\leq h_EC_{\mathbf{\vartheta}}\|v_h\|_{1,E}|w_h^*|_{1,E}\leq h_EC_{\mathbf{\vartheta}}\|v_h\|_{1,E}\vertiii{w_h^*}_E.
\end{multline*} 
Thus, summing over all the elements we have that 
\begin{align*}
T_4 \leq hC_{\mathbf{\vartheta}}\|v_h\|_{1,\O}\vertiii{w_h^*}\leq C_{\kappa,\mathbf{\vartheta}}h\|v_h\|_{1,\O}\|w^*\|_{1,\O}.
\end{align*} 
%We now note that, using that $v_h\in H_{0}^{1}(\O)$, $\PiK = \Pi^{E}$, and \eqref{eqrefgt2}, we have 
%\begin{align*}
%b(v_h &-\PiN v_h,\Pi w_h^*)=\int_{\O}\mathbf{\vartheta}(\mathbf{x})\cdot\nabla (v_h -\PiN v_h)\Pi w_h^*dx \\
%&=-\int_{\O}\div(\mathbf{\vartheta}(\mathbf{x})(\Pi w_h^*))(v_h -\Pi^{\nabla} v_h)\leq \|\mathbf{\vartheta}(\mathbf{x})(\Pi w_h^*)\|_{1,\O}\|v_h -\Pi^{\nabla} v_h\|_{0,\O}\\
%&\leq C_{\mathbf{\vartheta}}\|w_h^*\|_{1,\O}\|v_h -\Pi v_h\|_{0,\O}\leq C_{\mathbf{\vartheta}}h\vertiii{w_h^*}\|v_h \|_{1,\O}.
%\end{align*} 
%Therefore 
%\begin{equation}\label{eq:cuartacota}
%T_4\leq C_{\kappa,\mathbf{\vartheta}}h\|v_h\|_{1,\O}\|w^*\|_{1,\O}.
%\end{equation}
Therefore, from \eqref{eq:primeracota}, together with \eqref{eq:inf-supB2} and  \eqref{eq:estimacioninf-sup} we get that 
\begin{equation*}
 C_{\gamma,\kappa,\mathbf{\vartheta}}h\|v_h\|_{1,\O}\|w^*\|_{1,\O}+\mathcal{B}_h(v_h,w_h^*)\geq \mathcal{B}(v_h,w^*)\geq M_2\|v_h\|_{1,\O}\|w^*\|_{1,\O}. 
\end{equation*}
Then we have 
\begin{equation}
\label{eq:inf-sup_cf}
\mathcal{B}_h(v_h,w_h^*)\geq (M_2-C_{\gamma,\kappa,\mathbf{\vartheta}}h)\|v_h\|_{1,\O}\|w^*\|_{1,\O}. 
\end{equation}
On the other hand, using \eqref{eq:aux_forinf-sup}, \eqref{eqrefgt2}, together to \eqref{eq.elliptich}, we obtain
\begin{align}
\label{eq:inf-sup_cf2}
\|w^*\|_{1,\O}^2\geq\vertiii{w_h^*}^2=\sum_{E\in\CT_h}\vertiii{w_h^*}_{E}^2\geq\sum_{E\in\CT_h}\dfrac{1}{C_4}a^E(w_h^*,w_h^*)\geq \beta\|w_h^*\|_{1,\O}^2.
\end{align}

Hence, using \eqref{eq:inf-sup_cf} and \eqref{eq:inf-sup_cf2} it is straightforward to check that there exists $h_0 > 0$ such that for all $h < h_0$, the inf-sup condition \eqref{eq:inf-supBh} holds true. 
%The proof follows similar arguments of those in \cite[Lemma 5.7]{MR3460621}.
\end{proof}

With these results at hand, we are in position to conclude the following result that establishes the well posedness of \eqref{eq:virtual}
\begin{theorem}[Well posedeness]
For a given $f\in L^2(\Omega)$ and for $h$ sufficiently small, there exists  a unique $u_h\in V_h$ solution of problem \eqref{eq:virtual}. Moreover, there exists a positive constant $C>0$, independent of $h$,  such that $\|u_h\|_{1,\Omega}\leq C\|f\|_{0,\O}$.
\end{theorem}
\section{A priori error estimates}

In this section we derive a priori error estimates  for our virtual element method. To do this task, first we remark that we are considering \eqref{eq:derivative_stab} as stabilization of the method under the small edges approach. This motivates to adapt the results of \cite[Section 4.2]{BLR2017} for our problem, where the error estimates (more precisely the constants on each estimate) depend on the geometrical assumptions. 

%We begin with the following result, whose proof is the same as \cite[Lemma 3.5]{BLR2017}.
%\begin{lemma}
%For all $v_h\in V_h^E$, there exists a polynomial $\tilde{p}\in \mathbb{P}_{k+2}(E)$ such that $\Delta\tilde{p}=\Delta v_h$ and the following estimate holds
%\begin{equation*}
%|\tilde{p}|_{1,E}\leq C\|\Delta v_h\|_{0,E},
%\end{equation*}
%where $C$ is a positive constant independent of $h$.
%\end{lemma}

Now our goal is to derive error estimates for  our proposed virtual element method. We prove two error estimates: in one hand, we prove an $H^1$ error for the solution, and in the other an $L^2$ estimate. We begin with the $H^1$ estimate.
\begin{theorem}\label{teorem:error_1}
Let $u\in H^2(\O)\cap H^1(\O)$ and $u_h\in \Vh$ be the unique solutions of the continuous and discrete problems \eqref{eq:weak_state_equation} and \eqref{eq:virtual}, respectively. Then, for $h$ sufficiently small, there exists $C_{\gamma,\kappa,\mathbf{\vartheta}} > 0$, independent of h, such that
\begin{equation*}
\|u-u_h\|_{1,\O}\leq C_{\gamma,\kappa,\mathbf{\vartheta}} h\|f\|_{1,\O}.
\end{equation*}
\end{theorem}
\begin{proof}
From triangle inequality we have
\begin{equation}\label{eq:triangleeq}
\|u-u_h\|_{1,\O}\leq \|u-u_I\|_{1,\O}+\|u_h-u_I\|_{1,\O}.
\end{equation}
where we need to estimate  each of the terms on the right hand side. We observe that the first term  is directly controlled thanks to  Lemma \ref{estima2}. The main task is to bound the second term. With this in mind, we set $v_h:=u_h-u_I$ in  \eqref{eq:inf-supBh}. Then 
\begin{align}\nonumber
\widehat{\beta}\|u_h-u_I\|_{1,\O}\|w_h^*\|_{1,\O}&\leq \mathcal{B}_h(u_h,w_h^*)-\mathcal{B}_h(u_I,w_h^*)\\\nonumber
&=(f_h-f,w_h^*)_{0,\O}+\mathcal{B}_h(\PiN u-u_I,w_h^*)+\mathcal{B}(u-\PiN u,w_h^*)\\\label{eq:or_1}
&+\mathcal{B}(\PiN u,w_h^*)-\mathcal{B}_h(\PiN u,w_h^*),
\end{align}
where we have used, \eqref{eq:weak_state_equation}, \eqref{eq:virtual},  and elementary algebraic manipulations. Then, for the Cauchy–Schwarz inequality, we have
\begin{equation}\label{eq:or_11}\left(f_h-f,w_h^*\right)_{0,\O}\leq \|f_h-f\|_{0,\O}\|w_h^*\|_{0,\O}\leq Ch\|f\|_{1,\O}\|w_h^*\|_{1,\O}.\end{equation}
On the other hand, from \eqref{eq:acotamientoB} and \eqref{eq:cont_disc},  we obtain
\begin{multline*}\mathcal{B}_h(\PiN u-u_I,w_h^*)+\mathcal{B}(u-\PiN u,w_h^*)\\
\leq C_{\gamma,\kappa,\mathbf{\vartheta}}\left( \vertiii{\PiN u-u_I}\vertiii{w_h^*}+\|u-\PiN u\|_{1,\O}\|w_h^*\|_{1,\O}\right)\\
\leq C_{\gamma,\kappa,\mathbf{\vartheta}}\big( \left(\vertiii{\PiN u-u}+\vertiii{u-u_I}\right)\vertiii{w_h^*}+\|u-\PiN u\|_{1,\O}\vertiii{w_h^*}\big).
\end{multline*}
The following step is to bound each of the terms on the right-hand side of the previous estimate. To do this task, using the definition of $\vertiii{\cdot}$ (cf. \eqref{eq:triple}) and, again operating as in the proof of \cite[Theorem 4.5]{BLR2017}, we get
\begin{align*}
&\vertiii{u-u_I}^2=\sum_{E\in\CT_h}\vertiii{u-u_I}_E=\sum_{E\in\CT_h}\left\{a^E(\PiK(u-u_I),\PiK(u-u_I))\right.\\
&\left.+S^E(u-u_I-\overline{(u-u_I)},u-u_I-\overline{(u-u_I)})\right\}\\
&\leq C_{\kappa}\sum_{E\in\CT_h}\left\{|\PiK (u-u_I)|_{1,E}^2+S^E(u-u_I,u-u_I)\right\}\\
&\leq C_{\kappa}\sum_{E\in\CT_h}\left\{|u-u_I|_{1,E}^2+h_E|u-u_I|_{1,\partial E}^2\right\}\leq C_{\kappa}\sum_{E\in\CT_h}h_E^2|u|_{2,E}^2\leq C_{\kappa}h^2|u|_{2,\O}^2,
\end{align*}
where we have used a scaled trace inequality (see \cite[Lemma 6.1]{BLR2017}). Similarly, using also standard approximation results on polygons, we obtain  
\begin{align*}
\vertiii{\PiN u-u}\leq C_{\kappa}h|u|_{2,\O}.
\end{align*}
On the other hand, invoking Lemma \ref{lmm:bh}, we have 
$
\|u-\PiN u\|_{1,\O}\leq C\|u\|_{2,\O}
$.
Thus, from the above estimates we conclude that 
\begin{equation*}
\label{eq:or2}
\mathcal{B}_h(\PiN u-u_I,w_h^*)+\mathcal{B}(u-\PiN u,w_h^*)
\leq C_{\gamma,\kappa,\mathbf{\vartheta}}h\|u\|_{2,\O}\vertiii{w_h^*}.
\end{equation*} 
We now need a control for the term $\vertiii{w_h^*}$. To do this task, we  invoke Lemma \ref{lmm:technical1} with $\|w^*\|_{1,\O}=1$. Moreover,  there exists a constant $\widetilde{C}>0$ such that:
\begin{multline*}
\vertiii{w_h^*}^2\leq \sum_{E\in\CT_h}a_h^E(w_h^{*},w_h^{*})=a_h(w_h^*,w_h^*)=a(w^*,w_h^*)\leq C_\kappa\|w^*\|_{1,\O}\|w_h^*\|_{1,\O}\\\
\leq C_\kappa(1+\|w_h^*\|_{1,\O}^2)\leq (\widetilde{C}+1)\|w_h^*\|_{1,\O}^2.
\end{multline*} 
Then we have
\begin{equation}
\label{eq:or22}
\mathcal{B}_h(\PiN u-u_I,w_h^*)+\mathcal{B}(u-\PiN u,w_h^*)
\leq C_{\gamma,\kappa,\mathbf{\vartheta}}h\|u\|_{2,\O}\|w_h^*\|_{1,\O}.
\end{equation} 

%Then $\mathcal{B}_h(\PiN u-u_I,w_h^*)+\mathcal{B}(u-\PiN u,w_h^*)\leq Ch\|u\|_{2,\O}\|w_h^*\|_{1,\O}.$
Finally, we need to control the last term of \eqref{eq:or_1}. With this purpose, using that $\kappa(\mathbf{x})$ and $\gamma(\mathbf{x})$ are piecewise constant with respect to the meshes, together with the definitions of the bilinear forms $a_h^E(\cdot,\cdot)$, $b_h^E(\cdot,\cdot)$, $c_h^E(\cdot,\cdot)$, the properties of the virtual projectors and Lemma \ref{lmm:bh}, we have 
\begin{multline}
\label{eq:or_33}
\mathcal{B}(\PiN u,w_h^*)-\mathcal{B}_h(\PiN u,w_h^*)=b(\PiN u,w_h^*-\PiN w_h^*)\\
\leq C_{\mathbf{\vartheta}}|\PiK u|_{1,\O}\|w_h^*-\PiN w_h^*\|_{0,\O}\\
\leq C_{\mathbf{\vartheta}}|u|_{1,\O}\|w_h^*-\Pi w_h^*\|_{0,\O}\leq hC_{\mathbf{\vartheta}}\|u\|_{1,\O}\|w_h^*\|_{1,\O}.
\end{multline}
Then, gathering \eqref{eq:or_11}, \eqref{eq:or22}, \eqref{eq:or_33} and  replacing these estimates in \eqref{eq:or_1}, together with the approximation property given by Lemma \ref{estima2} in \eqref{eq:triangleeq}, we conclude the proof.
\end{proof}

\begin{theorem}\label{teorem:error_|||}
Let $u\in H^2(\O)\cap H^1(\O)$ and $u_h\in \Vh$ be the unique solutions of the continuous and discrete formulations \eqref{eq:weak_state_equation} and \eqref{eq:virtual}, respectively. Then, for $h$ sufficiently small, there exists $C_{\gamma,\kappa,\mathbf{\vartheta}} > 0$, independent of h, such that
\begin{equation*}
\vertiii{u-u_h}\leq C_{\gamma,\kappa,\mathbf{\vartheta}} h\|f\|_{1,\O}.
\end{equation*}
\end{theorem}
\noindent
\begin{proof}
From triangle inequality we have $\vertiii{u-u_h}\leq \vertiii{u-u_I}+\vertiii{u_h-u_I}$.
%\begin{equation*}\label{eq:triangleeq|||}
%\vertiii{u-u_h}\leq \vertiii{u-u_I}+\vertiii{u_h-u_I}.
%\end{equation*}
Let $v_h:=u_h-u_I\in V_h$. From \eqref{eqrefgt1}, we obtain
\begin{multline*}
\vertiii{u_h-u_I}^2=\sum_{E\in\CT_h}\vertiii{u_h-u_I}_E^2\\
\leq \sum_{E\in\CT_h}\dfrac{1}{c_1}a_h^E(u_h-u_I,v_h) 
\leq C\left(\sum_{E\in\CT_h}\left(a_h^E(u_h-,v_h)-a_h^E(u_I,v_h)\right)\right)\\
\leq C\left((f_h,v_h)-\sum_{E\in\CT_h}\left(b_h^E(u_h,v_h)+c_h^E(u_h,v_h)\right)\right.\\
\left.-\sum_{E\in\CT_h}\left[a_h^E(u_I-u_{\pi},v_h)+a^E(u_{\pi}-u,v_h)+a^E(u,v_h)\right]\right)\\
\leq C\left((f_h,v_h)-(f,v_h)+\sum_{E\in\CT_h}\left(b^E(u,v_h)-b_h^E(u_h,v_h)+c^E(u,v_h)-c_h^E(u_h,v_h)\right)\right.\\
\left.-\sum_{E\in\CT_h}\left[a_h^E(u_I-u_{\pi},v_h)+a^E(u_{\pi}-u,v_h)\right]\right).
%a_h^E(u_I-u_{\pi},v_h)-a_^E(u_{\pi}-u,v_h)-a^E(u,v_h)\\
\end{multline*}

Now we need to estimate each term on the right-hand side of the above estimate. With this goal in mind, we proceed as in the proof of Lemma \ref{lmm:disc_inf_sup} and Theorem \ref{teorem:error_1}, so we obtain the following estimates
\begin{align*}
(f_h,v_h)_{0,\O}-(f,v_h)_{0,\O}&\leq Ch\|f\|_{1,\O}\|v_h\|_{1,\O},\\
\sum_{E\in\CT_h}(b^E(u,v_h)-b_h^E(u_h,v_h))&\leq C_{\kappa,\mathbf{\vartheta}}h\|u\|_{2,\O}\|v_h\|_{1,\O},\\
\sum_{E\in\CT_h}(c^E(u,v_h)-c_h^E(u_h,v_h)) &\leq C_{\gamma,\kappa}h\|u\|_{1,\O}\|w_h\|_{1,\O},\\
\sum_{E\in\CT_h}(a_h^E(u_I-u_{\pi},v_h)+a^E(u_{\pi}-u,v_h)) &\leq C_{\kappa}h\|u\|_{2,\O}\vertiii{v_h}.
\end{align*}
Combining all these estimates we obtain
\begin{equation*}
\vertiii{u_h-u_I}^2\leq C_{\gamma,\kappa,\mathbf{\vartheta}}h(\|f\|_{1,\O}+\|u\|_{2,\O})\vertiii{u_h-u_I},
\end{equation*}
which concludes the proof.
\end{proof}

Our next goal is to improve the error estimate for the $L^2$-norm. This is contained in the following result.
\begin{theorem}
For $h$ sufficiently small, the following error estimate holds:
\begin{equation*}\label{eq:triangle_eq1}
\|u-u_h\|_{0,\O}\leq C_{\kappa,\mathbf{\vartheta},\gamma} h^2\|f\|_{0,\O}.
\end{equation*}
\end{theorem}
\begin{proof}
We procede with a standard duality argument. Let $\psi\in H^2(\O)\cap H_0^1(\O)$ be the solution of the following adjoint problem  
\begin{equation*}
\nabla \cdot (-\kappa(\mathbf{x}) + \mathbf{\vartheta}(\mathbf{x}) \psi) +\gamma(\mathbf{x}) \psi = u-u_h \quad \textrm{in}~\Omega, \qquad \psi = 0 \quad \textrm{on}~\partial\Omega.
\end{equation*}
On the other hand, let $\psi_I\in\Vh$ be its interpolant, satisfying the following error estimate
\begin{equation*}
\label{eq:aux_duality}
\|\psi-\psi_I\|_{1,\O}\leq Ch\|\psi\|_{2,\O}\leq Ch\|u-u_h\|_{0,\O}.
\end{equation*}
Now, proceeding as in the proof of \cite[Theorem 5.2]{MR3460621}, using elementary algebraic manipulations, together with the use of the properties of the virtual projector we have
\begin{multline}
\label{eq:L2_error}
\|u-u_h\|_{0,\O}^2=\mathcal{B}(u-u_h,\psi)=\mathcal{B}(u,\psi-\psi_I)+\mathcal{B}(u,\psi_I)-\mathcal{B}(u_h,\psi)\\
%&=\mathcal{B}(u,\psi-\psi_I)+(f-f_h,\psi_I)+\mathcal{B}_h(u_h,\psi_I)-\mathcal{B}(u_h,\psi)\\
%&=\mathcal{B}(u-u_h,\psi-\psi_I)+(f-f_h,\psi_I-\Pi_1\psi_I)+\mathcal{B}_h(u_h,\psi_I)-\mathcal{B}(u_h,\psi_I)\\
%&=\mathcal{B}(u-u_h,\psi-\psi_I)+(f-f_h,\psi_I-\Pi_1\psi_I)+\mathcal{B}_h(u_h-\PiN u,\psi_I)\\
%&+\mathcal{B}_h(\PiN u,\psi_I)-\mathcal{B}(u_h-\PiN u,\psi_I)-\mathcal{B}(\PiN u,\psi_I)\\
%&=(f-f_h,\psi_I-\Pi\psi_I)+\mathcal{B}(u-u_h,\psi-\psi_I)+\mathcal{B}_h(u_h-\PiN u,\psi_I)\\
%&-\mathcal{B}(u_h-\PiK u,\psi_I)+\mathcal{B}_h(\PiN u,\psi_I)-\mathcal{B}(\PiN u,\psi_I)\\
=(f-f_h,\psi_I-\Pi\psi_I)_{0,\O}+\mathcal{B}(u-u_h,\psi-\psi_I)+\mathcal{B}_h(u_h,\psi_I)-\mathcal{B}(u_h,\psi_I)
%&\leq \|f-f_h\|_{0,\O}\|\psi_I-\Pi_1\psi_I\|_{0,\O}+\|u-u_h\|_{1,\O}\|\psi-\psi_I\|_{1,\O}\\
%&+\vertiii{u_h-\PiN u}\vertiii{\psi_I-\PiN \psi_I}+\vertiii{u_h-\PiN u}\vertiii{\PiN \psi_I}\\
%&+\|u_h-\PiN u\|_{1,\O}\|\psi_I\|_{1,\O}+b(\PiN u,\psi_I-\PiN \psi_I)\\
%&\leq \|f-f_h\|_{0,\O}\|\psi_I-\Pi\psi_I\|_{0,\O}+\|u-u_h\|_{1,\O}\|\psi-\psi_I\|_{1,\O}\\
%&+\vertiii{u_h-\PiN u}\vertiii{\psi_I-\PiN \psi_I}+\vertiii{u_h-\PiN u}|\PiN \psi_I|_{1,\O}\\
%&+\|u_h-\PiN u\|_{1,\O}\|\psi_I\|_{1,\O}+\C_{\vartheta}|\PiN u|_{1,\O}\|\psi_I-\PiN \psi_I\|_{0,\O},
\end{multline}
where the last inequality of the previous estimation is obtained by proceeding in the same way as in the proof of Theorem \ref{teorem:error_1}. the following step is to bound all the terms on the right-hand side of \eqref{eq:L2_error}. We note that by the properties used above it can be shown that 
\begin{align*}
(f-f_h,\psi_I-\Pi\psi_I)_{0,\O}&\leq h^2|f|_{1,\O}\|\psi\|_{2,\O}\leq h^2|f|_{1,\O}\|u-u_h\|_{0,\O};\\
\mathcal{B}(u-u_h,\psi-\psi_I)&\leq C_{\gamma,\kappa,\mathbf{\vartheta}} h^2\|u\|_{2,\O}\|\psi\|_{2,\O}\leq C_{\gamma,\kappa,\mathbf{\vartheta}}h^2\|u\|_{2,\O}\|u-u_h\|_{0,\O}.
\end{align*}
For the last term we have that:
 \begin{align*}
\mathcal{B}_h(u_h,\psi_I)-\mathcal{B}(u_h,\psi_I)=\sum_{E\in\CT_h}\left(a_h^E(u_h,\psi_I)-a^E(u_h,\psi_I)+b_h^E(u_h,\psi_I)-b^E(u_h,\psi_I)\right.\\\left.+c_h^E(u_h,\psi_I)-c^E(u_h,\psi_I)\right).
\end{align*}
Proceeding according to the proofs of Lemmas \ref{lmm:technical1} and \ref{lmm:disc_inf_sup}, the following estimates can be obtained
\begin{align*}
\sum_{E\in\CT_h}(a_h^E(u_h,\psi_I)-a^E(u_h,\psi_I))&\leq \sum_{E\in\CT_h}C_\kappa(\vertiii{u_h-\PiK u}_E\vertiii{\psi_I-\PiK \psi}_E\\
&\leq C_\kappa h^2\|u\|_{2,\O}\|u-u_h\|_{0,\O}.\\
\sum_{E\in\CT_h}(b_h^E(u_h,\psi_I)-b^E(u_h,\psi_I)&\leq C_\vartheta h^2\|u\|_{2,\O}\|u-u_h\|_{0,\O}.\\
\sum_{E\in\CT_h}(c_h^E(u_h,\psi_I)-c^E(u_h,\psi_I))&\leq \sum_{E\in\CT_h}C_\gamma\|u_h-\Pi^Eu\|_{0,E}\|\psi_I-\Pi^E \psi\|_{0,E}\\
&\leq C_\gamma h^2\|u\|_{2,\O}\|u-u_h\|_{0,\O}.
\end{align*}
Thus, by combining all the estimates obtained, it is concluded that
\begin{align*}
\|u-u_h\|_{0,\O}^2\leq C_{\kappa,\vartheta,\gamma}h^2(\|f\|_{1,\O}+\|u\|_{2,\O})\|u-u_h\|_{0,\O},
\end{align*}
which makes it possible to conclude the proof.

\end{proof}

\section{The eigenvalue problem}
\label{sec:spectral}
As a consequence of the previous results, now we are interested in the natural extension of considering the associated
eigenvalue problem. It is important to take into account that the spectral problem is non symmetric and hence, the eigenvalues
associated to the solution operator are expectable to be complex (see \cite{MR3133493} for instance). Also, we claim that from now and on, the bilinear form $c(\cdot,\cdot)$ is no longer considered. This implies that our eigenvalue problem is set for the difussion-convection problem as the one analyzed in \cite{MR3133493}. 
\subsection{Spectral continuous problem}

The definition of the bilinear form $\mathcal{B}(\cdot,\cdot)$ given in  \eqref{eq:bilinear_cont}, must be modified for the eigenvalue problem  considering now  complex conjugated test functions. Moreover, since we will consider the difussion-convection spectral problem, the bilinear form $c(\cdot,\cdot)$ is no longer needed. Hence,  the eigenvalue problem reads as follows:  Find $\lambda\in\mathbb{C}$ and $0\neq u\in H_0^1(\O)$
such that  
\begin{equation}
\label{eq:spectral1}
\CB_{ab}(u,v)=\lambda  d(u,v)  \quad \forall v \in H_{0}^{1}(\Omega),
\end{equation}
where, $\CB_{ab}(\cdot,\cdot)$ is the bilinear form defined by 
\begin{equation*}
\CB_{ab}(w,v):=\int_{\Omega}\kappa(\mathbf{x}) \nabla w\cdot\nabla \bar{v} +\int_{\Omega}( \mathbf{\vartheta}(\mathbf{x})\cdot \nabla w)\bar{v}\quad\forall w,v\in H_{0}^{1}(\Omega),
\end{equation*}
whereas $d(\cdot,\cdot)$ is the bilinear form defined by 
$d(w,v):=(w, v)_{0,\O}.$
Clearly in this context, the space $H_0^1(\O)$ must be understood as a complex Hilbert space. Moreover,  we denote by $\bar{v}$  the conjugate of $v$.

Under the assumption that $ \mathbf{\vartheta}(\mathbf{x})$ is divergence free, it is easy to check that there exists $\underline{\beta}>0$ such that
\begin{equation*}
\label{eq:inf-supBab}
\displaystyle\sup_{v\in H_0^1(\Omega)}\frac{\mathcal{B}_{ab}(w,v)}{\|v\|_{1,\O}}\geq \underline{\beta}\|w\|_{1,\O}\quad \forall w\in H_0^1(\O).
\end{equation*}
This allows us to introduce the solution operator
$T$, defined as follows
$$T:H_0^1(\O)\rightarrow H_0^1(\O),\quad f\mapsto Tf=\widetilde{u},$$
where $\widetilde{u}\in H_0^1(\O)$ is the solution of the following source problem
\begin{equation}
\label{eq:source}
\CB_{ab}(\widetilde{u},v)=\mathcal{F}(v)\quad \forall v \in H_{0}^{1}(\Omega),
\end{equation}
where $\mathcal{F}$ is the functional defined on \eqref{eq:functionaF}. Let us remark that 
 $T$ is well defined and  compact   due to the compact inclusion of $H^2(\O)$ onto $H^1(\O)$. On the other hand, we observe that
$(\lambda,u)\in \mathbb{C}\times H_0^1(\O)$ is a solution of  \eqref{eq:spectral1} if and only if $(\mu,u)\in\mathbb{C}\times H_0^1(\O)$ is an eigenpair of $T$.

Since we have the additional regularity described in \eqref{eq:regularity}, the following spectral characterization of $T$ holds.
\begin{lemma}[Spectral Characterization of $T$]
The spectrum of $T$ is such that $\sp(T)=\{0\}\cup\{\mu_k\}_{k\in\mathbb{N}}$, where $\{\mu_k\}_{k\in\mathbb{N}}$ is a sequence of complex eigenvalues that converge to zero, according to their respective multiplicities.
\end{lemma}

\subsection{Spectral discrete problem}
Now we introduce the VEM discretization of problem \eqref{eq:spectral1}. To do this task, we requiere the global space $V_h$
defined in \eqref{eq:globa_space} together with the assumptions introduced in Section \ref{sec:virtual}.

The spectral problem reads as follows: find $\lambda_h\in\mathbb{C}$ and $0\neq u_h\in V_h$ such that
\begin{equation*}
\label{eq:spectral_disc}
\CB_{ab,h}(u_h,v_h)=\lambda_h d_h(u_h,v_h) \quad \forall v_h \in V_h,
\end{equation*}
where $\CB_h(\cdot,\cdot)$ is the bilinear form defined in \eqref{eq:bilineal_form_B_split_{h}} and $d_h(\cdot,\cdot)$ is the bilinear form defined by $d_h(w_h,v_h):=(\Pi w_h,\Pi v_h)_{0,\O}.$

Adapting the proof of Lemma \ref{lmm:disc_inf_sup} we prove the existence  of  a constant $\underline{\widehat{\beta}}>0$ such that, for all $h<h_0$ there holds
\begin{equation}\label{eq:inf-supBhab}
\displaystyle\sup_{w_h\in V_h}\frac{\mathcal{B}_{ab,h}(v_h,w_h)}{\|w_h\|_{1,\O}}\geq\underline{\widehat{\beta}}\|v_h\|_{1,\Omega}\quad\forall v_h\in V_h.
\end{equation}

This allows us to introduce the discrete counterpart of $T$, 
namely $T_h$, that is defined by 
$$T_h:H_0^1(\O)\rightarrow V_h,\quad f\mapsto T_hf=\widetilde{u}_h,$$
where $\widetilde{u}_h\in V_h$ is the solution of the following source problem
\begin{equation}
\label{eq:source_discrete}
\CB_{ab,h}(\widetilde{u}_h,v_h)=d_h(f,v_h)\quad \forall v_h \in V_h.
\end{equation}

The main goal is to analyze the convergence of the method and derive error estimates for the eigenvalues and eigenfunctions. Due to the compactness of $T$, the convergence of the eigenvalues is derived from the classic theory of \cite{BO}.
Since the eigenvalue problem is nonsymmetric, it is important to consider  the associated adjoint problem. To do this task, let us denote by $T^{*}$ and $T_h^{*}$
the adjoint operators of $T$ and $T_h$ respectively, both defined by $T^*f=\widetilde{u}^*$ and $T_h^*f=\widetilde{u}_h^*$, where $\widetilde{u}^*$ and $\widetilde{u}_h^*$ are the solution of the following problems
\begin{equation*}
\CB_{ab}(v,\widetilde{u}^*)=d(f,v)\quad\forall v\in H_0^1(\O)
\quad
\text{and}
\quad
\CB_{ab,h}(v,\widetilde{u}_h^*)=d_h(f,v_h)\quad\forall v\in  V_h.
\end{equation*}
With these operators at hand, our first task is to prove the convergence in norm of $T_h$ to $T$ and the adjoints counterparts
$T^*_h$ to $T^*$ as $h$ goes to zero. We begin our analysis with the following result.

\begin{lemma}
\label{lmm:conv1}
Let $f\in L^2(\Omega)$ be such that $\widetilde{u}:=Tf$ and $\widetilde{u}_h:=T_hf$. Then, there exists a positive constant $C_{\kappa,\mathbf{\vartheta}}$, independent of $h$, such that 
\begin{equation*}
\|(T-T_h)f\|_{1,\Omega}\leq C_{\kappa,\mathbf{\vartheta}}h\|f\|_{0,\O}.%+|\tilde{u}-\tilde{u}_I|_{1,\Omega}+|\tilde{u}-\tilde{u}_{\pi}|_{1,\Omega}\right).
\end{equation*}
\end{lemma}
\begin{proof}
To derive this result we need to invoke the inf-sup condition \eqref{eq:inf-supBhab} Similarly as the proof of Theorem \ref{teorem:error_1}, 
if $\widetilde{u}=:Tf$ and $\widetilde{u}_h:=T_hf$, from triangle inequality we have 
\begin{equation*}
\|\widetilde{u}-\widetilde{u}_h\|_{1,\O}\leq \|\widetilde{u}-\widetilde{u}_I\|_{1,\O}+\|\widetilde{u}_I-\widetilde{u}_h\|_{1,\O},
\end{equation*}
where $\widetilde{u}_I$ represents the interpolation of $\widetilde{u}$. Setting $v_h=\widetilde{u}_I-\widetilde{u}_h$ in \eqref{eq:inf-supBhab} and using \eqref{eq:source}
and \eqref{eq:source_discrete} we obtain
\begin{multline*}
\underline{\widehat{\beta}}\|\widetilde{u}_I-\widetilde{u}_h\|_{1,\O}\|v_h\|_{1,\O}\leq \underbrace{[d_h(f_h,v_h)-d(f,v_h)]}_{A}\\
+\underbrace{[\CB_{ab}(\widetilde{u}-\Pi^{\nabla}\widetilde{u},v_h)-\CB_{ab,h}(\widetilde{u}_I-\Pi^{\nabla}\widetilde{u},v_h)]}_{B}+\underbrace{[\CB_{ab}(\Pi^{\nabla}\widetilde{u},v_h)-\CB_{ab,h}(\Pi^{\nabla}\widetilde{u},v_h)]}_{C},
\end{multline*}
where each contribution $A$, $B$ and $C$  are estimated analogously as in the proof of Theorem \ref{teorem:error_1} and \eqref{eq:regularity}.
%Let $f\in L^2(\Omega)$ and $\tilde{u}_I\in V_h$. From the definitions of $T$ and $T_h$, and triangle inequality, we have
%\begin{equation*}
%\|(T-T_h)f\|_{1,\Omega}\leq \|\tilde{u}-\tilde{u}_I\|_{1,\Omega}+\|\tilde{u}_I-\tilde{u}_h\|_{1,\Omega}.
%\end{equation*}
%Now, invoking the ellipticity of $\CB_{ab,h}(\cdot,\cdot)$ on $V_h$ and defining $v_h=\tilde{u}_h-\tilde{u}_I$ we have
%\begin{multline*}
%\alpha\|v_h\|_{1,\Omega}^2\leq \CB_{ab,h}(v_h,v_h)=\CB_{ab,h}(\tilde{u}_h-\tilde{u}_I.v_h)\\
%=\CB_{ab,h}(\tilde{u}_h,v_h)-\CB_{ab,h}(\tilde{u}_I,v_h)
%=d_h(f,v_h)-a_h(\tilde{u}_I,v_h)-b_h(\tilde{u}_I,v_h)\\
%=d_h(f,v_h)-a_h(\tilde{u}_I-\tilde{u}_{\pi},v_h)-a_h(\tilde{u}_{\pi},v_h)-b_h(\tilde{u}_I,v_h)\\
%=d_h(f,v_h)-b_h(\tilde{u}_I,v_h)-\sum_{E\in\CT_h}\left\{a_h^E(\tilde{u}_I-\tilde{u}_{\pi},v_h)-a_h^E(\tilde{u}_{\pi}-\tilde{u},v_h)+a_h^E(\tilde{u},v_h)\right\}\\
%=d_h(f,v_h)-b_h(\tilde{u}_I,v_h)-a_h(\tilde{u},v_h)-\sum_{E\in\CT_h}\left\{a_h^E(\tilde{u}_I-\tilde{u}_{\pi},v_h)-a_h^E(\tilde{u}_{\pi}-\tilde{u},v_h)\right\}
%\end{multline*}
\end{proof}

Therefore we have proved that the discrete solution operator $T_h$ converge in norm to the continuous one $T$, as $h$ goes to zero.
%\begin{lemma}
%\label{lmm:conv:normT}
%There exists a positive constant $C_{\kappa,\mathbf{\vartheta}}$, independent of $h$, such that
%\begin{equation*}
%\|T-T_h\|_{1,\O}\leq C_{\kappa,\mathbf{\vartheta}}h.
%\end{equation*}
%\end{lemma}
%\begin{proof}
%Let $f\in L^2(\Omega)$ be such that $\|f\|_{0,\O}=1$. Let $\widetilde{u}$ be the solution of \eqref{eq:source} and 
%$\widetilde{u}_h$ be the solution of \eqref{eq:source_discrete}, such that $Tf=\widetilde{u}$ and $T_hf=\widetilde{u}_h$.
%\end{proof}

The previous result also holds for the adjoint operators $T_h^*$ and $T^*$. Since the proof is analogous, we skip the details.
\begin{lemma}
\label{eq:adjoint_diff}
There exists a positive constant $C_{\kappa,\mathbf{\vartheta}}$, independent of $h$, such that
\begin{equation*}
\|T^*-T_h^*\|_{1,\O}\leq C_{\kappa,\mathbf{\vartheta}}h\|f\|_{0,\O}.
\end{equation*}
\end{lemma}
As a direct consequence of the above two lemmas, standard results on spectral approximation can be used (see \cite{MR2652780,MR0203473}).

We present as a consequence of the above, that the proposed method does not introduce spurious eigenvalues (see \cite{MR0203473}).
\begin{theorem}
Let $G \subset \mathbb{C}$  be an open set containing $\sp(T)$. Then, there exists $h_0>0$ such that $\sp(T_h)\subset G$ for all $h<h_0$.
\end{theorem}

\subsection{Error estimates}
The goal of this section is deriving error estimates for the eigenfunctions and eigenvalues. We first recall the definition of spectral projectors. Let $\mu$ be a nonzero isolated eigenvalue of $T$ with algebraic multiplicity $m$ and let $\Gamma$
be a disk of the complex plane centered in $\mu$, such that $\mu$ is the only eigenvalue of $T$ lying in $\Gamma$ and $\partial\Gamma\cap\sp(T)=\emptyset$. With these considerations at hand, we define the spectral projections of $E$ and $E^*$, associated to $T$ and $T^*$, respectively, as follows:
\begin{enumerate}
\item The spectral projector of $T$ associated to $\mu$ is $\displaystyle E:=\frac{1}{2\pi i}\int_{\partial\Gamma} (zI-T)^{-1}\,dz;$
\item The spectral projector of $T^*$ associated to $\bar{\mu}$ is $\displaystyle E^*:=\frac{1}{2\pi i}\int_{\partial\Gamma} (zI-T^*)^{-1}\,dz,$
\end{enumerate}
where $I$ represents the identity operator. Let us remark that $E$ and $E^*$ are the projections onto the generalized eigenvector $R(E)$ and $R(E^*)$, respectively. 

A consequence of Lemma \ref{lmm:conv1} is that there exist $m$ eigenvalues, which lie in $\Gamma$, namely $\mu_h^{(1)},\ldots,\mu_h^{(m)}$, repeated according their respective multiplicities, that converge to $\mu$ as $h$ goes to zero. With this result at hand, we introduce the following spectral projection
\begin{equation*}
E_h:=\frac{1}{2\pi i}\int_{\partial\Gamma} (zI-T_h)^{-1}\,dz,
\end{equation*}
which is a projection onto the discrete invariant subspace $R(E_h)$ of $T$, spanned by the generalized eigenvector of $T_h$ corresponding to 
 $\mu_h^{(1)},\ldots,\mu_h^{(m)}$.

Now we recall the definition of the \textit{gap} $\hdel$ between two closed
subspaces $\CM$ and $\CN$ of $L^2(\O)$:
$$
\hdel(\CM,\CN)
:=\max\big\{\delta(\CM,\CN),\delta(\CN,\CM)\big\}, \text{ where } \delta(\CM,\CN)
:=\sup_{\underset{\left\|x\right\|_{0,\O}=1}{x\in\CM}}
\left(\inf_{y\in\CN}\left\|x-y\right\|_{0,\O}\right).
$$
We end this section proving error estimates for the eigenfunctions and eigenvalues.
\begin{theorem}
\label{thm:errors1}
There exists $C_{\kappa,\mathbf{\vartheta}}>0$ such that
\begin{equation*}
\hdel(R(E),R(E_h))\leq C_{\kappa,\mathbf{\vartheta}}h\quad\text{and}\quad
|\mu-\mu_h|\leq C_{\kappa,\mathbf{\vartheta}}h^2.
\end{equation*}
\end{theorem}
\begin{proof}
The proof of the gap between the eigenspaces is a direct consequence of the convergence in norm between $T$ and $T_h$ as $h$ goes to zero.
We focus on the double order of convergence for the eigenvalues. Let $\{u_k\}_{k=1}^m$ be such that $Tu_k=\mu u_k$, for $k=1,\ldots,m$. A dual basis for $R(E^*)$ is $\{u_k^*\}_{k=1}^m$. This basis satisfies $\mathcal{B}_{ab}(u_k,u_l^*)=\delta_{k.l},$
where $\delta_{k.l}$ represents the Kronecker delta.

On the other hand, the following identity holds
\begin{equation*}
|\mu-\widehat{\mu}_h|\leq\frac{1}{m}\sum_{k=1}^m|\langle(T-T_h)u_k,u_k^* \rangle|+C\|(T-T_h)|_{R(E)} \| \|(T^*-T_h^*)|_{R(E^*)}\|,
\end{equation*}
where $C>0$. We observe that the bound for the last two terms in the inequality above, are directly obtained from Lemma \ref{lmm:conv1}. Hence, our task is  bound the remaining first term. In order to do this,  the following identity can be obtained
\begin{multline*}
|\langle(T-T_h)u_k,u_k^* \rangle|=\mathcal{B}_{ab}((T-T_h)u_k,u_k^*)\\
=\mathcal{B}_{ab}((T-T_h)u_k,u_k^*-v_h)+\mathcal{B}_{ab}(Tu_k,v_h))-\mathcal{B}_{ab}(T_hu_k,v_h)\\
=\underbrace{\mathcal{B}_{ab}((T-T_h)u_k,u_k^*-v_h)}_{\textbf{(I)}}+\underbrace{(d(u_k,v_h)-d_h(u_k,v_h))}_{\textbf{(II)}}\\
+\underbrace{(\mathcal{B}_{ab,h}(T_hu_k,v_h)-\mathcal{B}_{ab}(T_hu_k,v_h)}_{\textbf{(III)}},
\end{multline*}
for all $v_h\in V_h$. Now we bound each of the contributions \textbf{(I)}, \textbf{(II)} and \textbf{(III)}. For \textbf{(I)}
first we set $v_h:=(u_k^*)_I$. Then we have
\begin{multline}
\label{eq:termI}
\textbf{(I)}=\mathcal{B}_{ab}((T-T_h)u_k,u_k^*-(u_k^*)_I)
\leq C_{\kappa,\mathbf{\vartheta}}\|(T-T_h)u_k\|_{1,\O}\|u_k^*-(u_k^*)_I\|_{1,\O}\\
%= \|(T-T_h)u_k\|_{1,\O}\inf_{r_h\in V_h}\|u_k^*-r_h\|_{1,\O}
%\leq \|(T-T_h)u_k\|_{1,\O}\|u_k^*-\mathcal{P} u_k^*\|_{1,\O}\\
\leq C_{\kappa,,\mathbf{\vartheta}}|(T-T_h)u_k|_{1,\O}|u_k^*-(u_k^*)_I|_{1,\O}
\leq C_{\kappa,,\mathbf{\vartheta}}h^{2}|u_k|_{1,\O}|u_k^*|_{1,\O},
\end{multline}
where we have used approximation properties for $(u_k^*)_I$ and the convergence in norm given by Lemma \ref{lmm:conv:normT}.
Now to control term \textbf{(II)}, we notice that
\begin{multline}
\label{eq:termII}
(\mathbf{II})\leq \sum_{E\in\CT_h}\|u_k-\Pi^Eu_k\|_{0,E}\|(u_k^*)_I-\Pi^E(u_k^*)_I\|_{0,E}\\
\leq C\sum_{E\in\CT_h}h_E |u_k|_{1,E}(\|u_k^*-(u_k^*)_I\|_{0,E}+\|u_k^*-\Pi^Eu_k^*\|_{0,E})\\
\leq Ch^2\|u_k\|_{1,\O}\|u_k^*\|_{1,\O}.
\end{multline}
For the term \textbf{(III)} we define $w_h:=T_hu_k$. Then
\begin{multline*}
\textbf{(III)}=\sum_{E\in\CT_h}[\mathcal{B}_{ab,h}^E(w_h,(u_k^*)_I)-\mathcal{B}_{ab}^E(w_h,(u_k^*)_I)]\\
=\underbrace{\sum_{E\in\CT_h} [a_h^E(w_h,(u_k^*)_I)-a^E(w_h,(u_k^*)_I)]}_{A_1}+\underbrace{\sum_{E\in\CT_h}[b_h^E(w_h,(u_k^*)_I)-b^E(w_h,(u_k^*)_I)]}_{A_2},
\end{multline*}
where we need to control the terms $A_1$ and $A_2$. For $A_1$ we have
\begin{multline}\label{eq:intermedio}
a_h^E(w_h,(u_k^*)_I)-a^E(w_h,(u_k^*)_I)\leq C_{\kappa}\left(\vertiii{w_h-\Pi^{\nabla,E}w_h}_E\vertiii{(u_k^*)_I-\Pi^{\nabla,E}(u_k^*)_I}_E\right.\\
 \left.+|w-\Pi^{\nabla,E}w_h|_{1,E}|(u_k^*)_I-\Pi^{\nabla,E}(u_k^*)_I|_{1,E}\right).
 \end{multline}
 
Now our task is to estimate each contribution on the right hand side of \eqref{eq:intermedio}. Notice that thanks to the additional regularity of the eigenfunction we have
 \begin{equation}
 \label{eq:cotas_semis}
 |w-\Pi^{\nabla,E}w_h|_{1,E}|(u_k^*)_I-\Pi^{\nabla,E}(u_k^*)_I|_{1,E}\leq Ch_E^2|w_h|_{2,E},
 \end{equation}
 together with the following estimate
 \begin{equation}
 \label{eq:cota_tripleI}
  \vertiii{w_h-\Pi^{\nabla,E}w_h}_E\leq Ch_E|w_h|_{2,E}.
 \end{equation}
 
Observe that using the definition of $\vertiii{\cdot}_E$ we obtain
\begin{multline*}
\vertiii{(u_k^*)_I-\Pi^{\nabla,E}(u_k^*)_I}_E^2=S^E((u_k^*)_I-\Pi^{\nabla,E}(u_k^*)_I,(u_k^*)_I-\Pi^{\nabla,E}(u_k^*)_I)\\
=h_E|(u_k^*)_I-\Pi^{\nabla,E}(u_k^*)_I|_{1,\partial E}^2\\
\leq h_E|u_k^*-(u_k^*)_I|_{1,\partial E}^2+h_E|u_k^*-\Pi^{\nabla,E}u_k^*|_{1,\partial E}^2+h_E|\Pi^{\nabla,E}(u_k^*-u_k^*)_I|_{1,\partial E}^2.
%\leq h_E|u_k^*-(u_k^*)_I|_{1,\partial E}^2+h_E|u_k^*-\Pi^{\nabla,E}u_k^*|_{1,\partial E}^2+h_E|\Pi^{\nabla,E}(u_k^*-u_k^*)_I|_{1,\partial E}^2
 \end{multline*}
We note that proceeding as in the proof of  \cite[Theorem 4.5]{BLR2017}, and using that $\PiK\in \P_{1}(E)$ which is stable in $|\cdot|_{1,E}$,  we have  
\begin{equation}
\label{eq:cota_tripleII}
\vertiii{(u_k^*)_I-\Pi^{\nabla,E}(u_k^*)_I}_E\leq Ch_E|u_k^*|_{2,E}.
\end{equation}
Thus, gathering \eqref{eq:cota_tripleI},\eqref{eq:cota_tripleII}, and \eqref{eq:cotas_semis},  an estimate for \eqref{eq:intermedio} is 
\begin{equation}
\label{eq:cotaA1_final}
a_h^E(w_h,(u_k^*)_I)-a^E(w_h,(u_k^*)_I)\leq C_\kappa h_E^2|u_k|_{2,E}|u_k^*|_{2,E}.
\end{equation}

 On the other hand, proceeding as in the proof of Lemma \ref{lmm:disc_inf_sup},  we obtain the follwing estimate for $A_2$
\begin{equation}
\label{eq:cotaA2_final}
b_h^E(w_h,(u_k^*)_I)-b^E(w_h,(u_k^*)_I)\leq C_{\mathbf{\vartheta}}h_E^2|w_h|_{2,E}^2|u_k^*|_{2,E}.
 \end{equation}
Hence, summing over all $E\in\CT_h$ in \eqref{eq:cotaA1_final} and \eqref{eq:cotaA2_final}, we obtain the desire bounds of $A_1$ and $A_2$, respectively, implying that 
\begin{equation}
\label{eq:termIII}
(\mathbf{III})\leq C_{\mathbf{\vartheta}}h^{2}|u_k|_{1,\O}|u_k^*|_{1,\O}.
\end{equation}
 Finally, combining \eqref{eq:termI}, \eqref{eq:termII}, and \eqref{eq:termIII} we conclude the proof.
\end{proof}

\section{Numerical experiments}
\label{sec:numerics}
In this section we report some numerical tests in order to explore with 
computational evidence the performance of the numerical method proposed in
our paper. 

We have implemented a MATLAB code  for the tests considering the lowest order
VEM for our  space \eqref{Vk}. We divide this section into two subsections: the first one reports numerical tests for the source problem where
we are interested in the computation of errors and convergence rates for the $L^2$ norm and the $H^1$ seminorm. Since the VEM solution $u_h$ is not explicitly known inside the elements, we compare $u$ with the $L^2$-projection of $u_h$ on $\P_1$, the same consideration will be used for the $H^1$ seminorm. ie, $\|u-\Pi u_h\|_{0,\O}$ for $L^2$ and $|u-\Pi^{\nabla} u_h|_{1,\O}$ for $H^1$.

 The second part is 
dedicated to the eigenvalue problem where our task is to assess the performance of the method on the approximation of the spectrum of $T$. 

Along all our experimental section we consider meshes that only 
satisfy assumption $\mathbf{A1}$.
In Figure~\ref{FIG:meshes}, we present plots of the polygonal meshes that
we will consider for our tests. We note that the family of polygonal meshes $\CT_h^1$
have been obtained by gluing two different polygonal meshes at $y=0.6$.
It can be seen that very small edges compared with the element
diameter appears on the interface of the resulting mesh.

On the other hand, the family of polygonal meshes $\CT_h^2$ have been obtained
from a triangular mesh with an additional point on
each edge as a new degree of freedom which has been moved
to a distance $h_e^{2}$ from one vertex and $(h_{e}-h_{e}^{2})$
from the other. We remark  that this family satisfy $\mathbf{A1}$
but fail to satisfy the usual assumption that distance between
any two of its vertices is greater than or equal to $Ch_{\E}$ for each polygon,
since the length of the smallest edge is $h_e^2$, while the diameter
of the element is bounded above by a multiple of $h_e$.
The refinement level for the meshes will be denoted by $N$,
which corresponds to the number of subdivisions in the abscissae. 

\begin{figure}[H]
\begin{center}
\begin{minipage}{13cm}
%\centering\includegraphics[height=4.5cm, width=4.5cm]{tablasymallas/thao1.eps}\\
\centering\includegraphics[height=4.0cm, width=4.0cm]{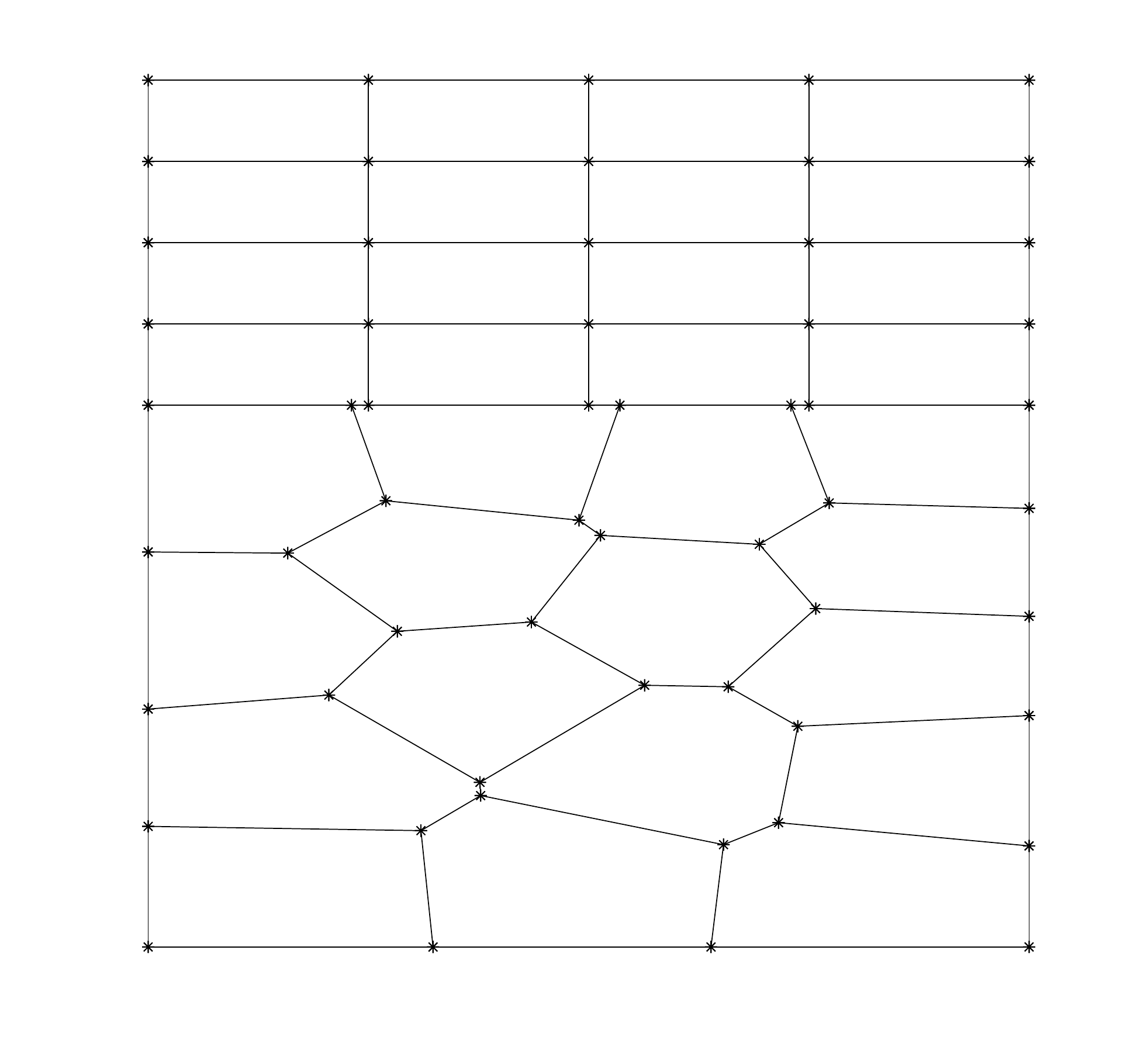}
\centering\includegraphics[height=4.0cm, width=4.0cm]{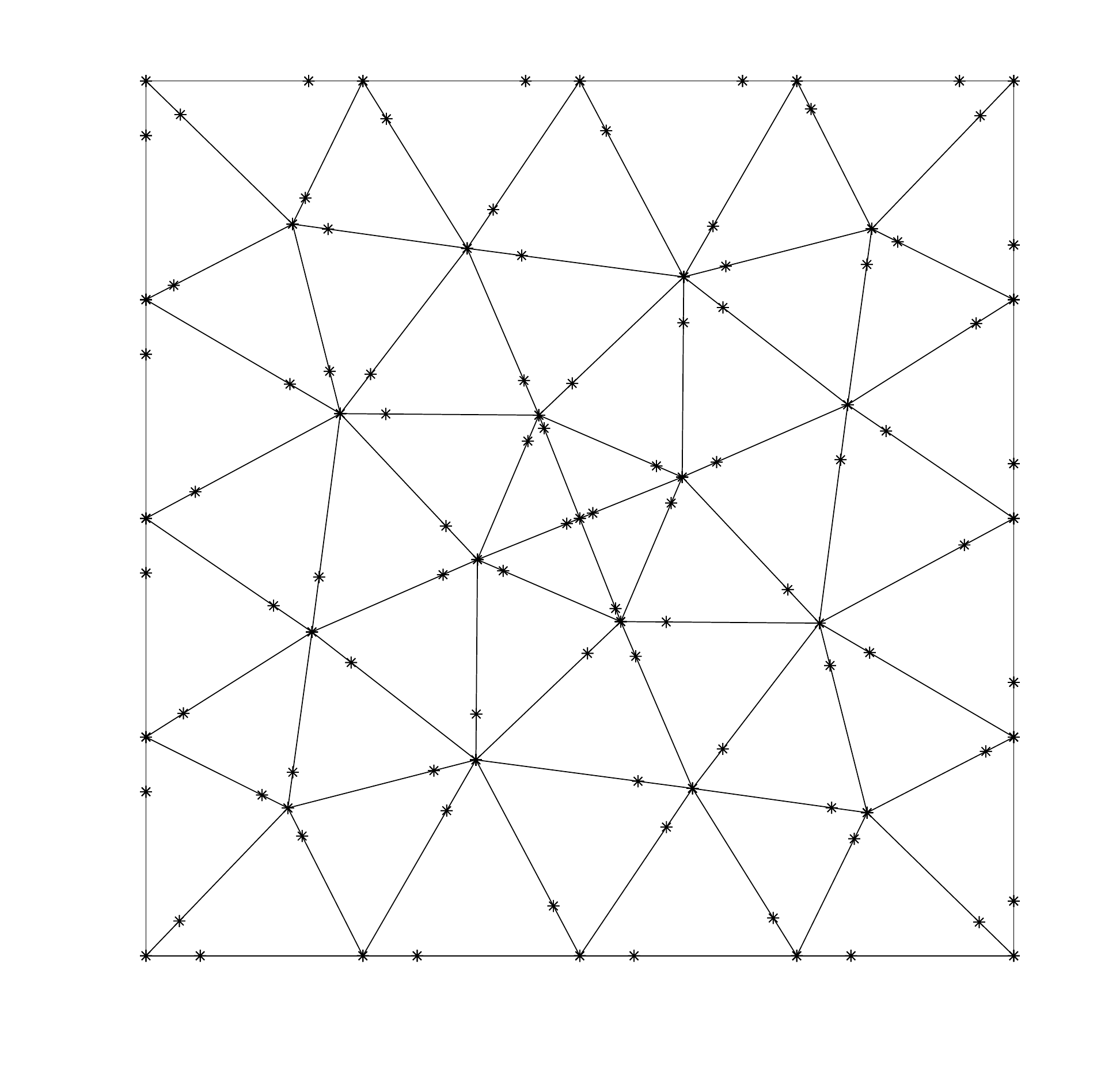}
\centering\includegraphics[height=4.0cm, width=4.0cm]{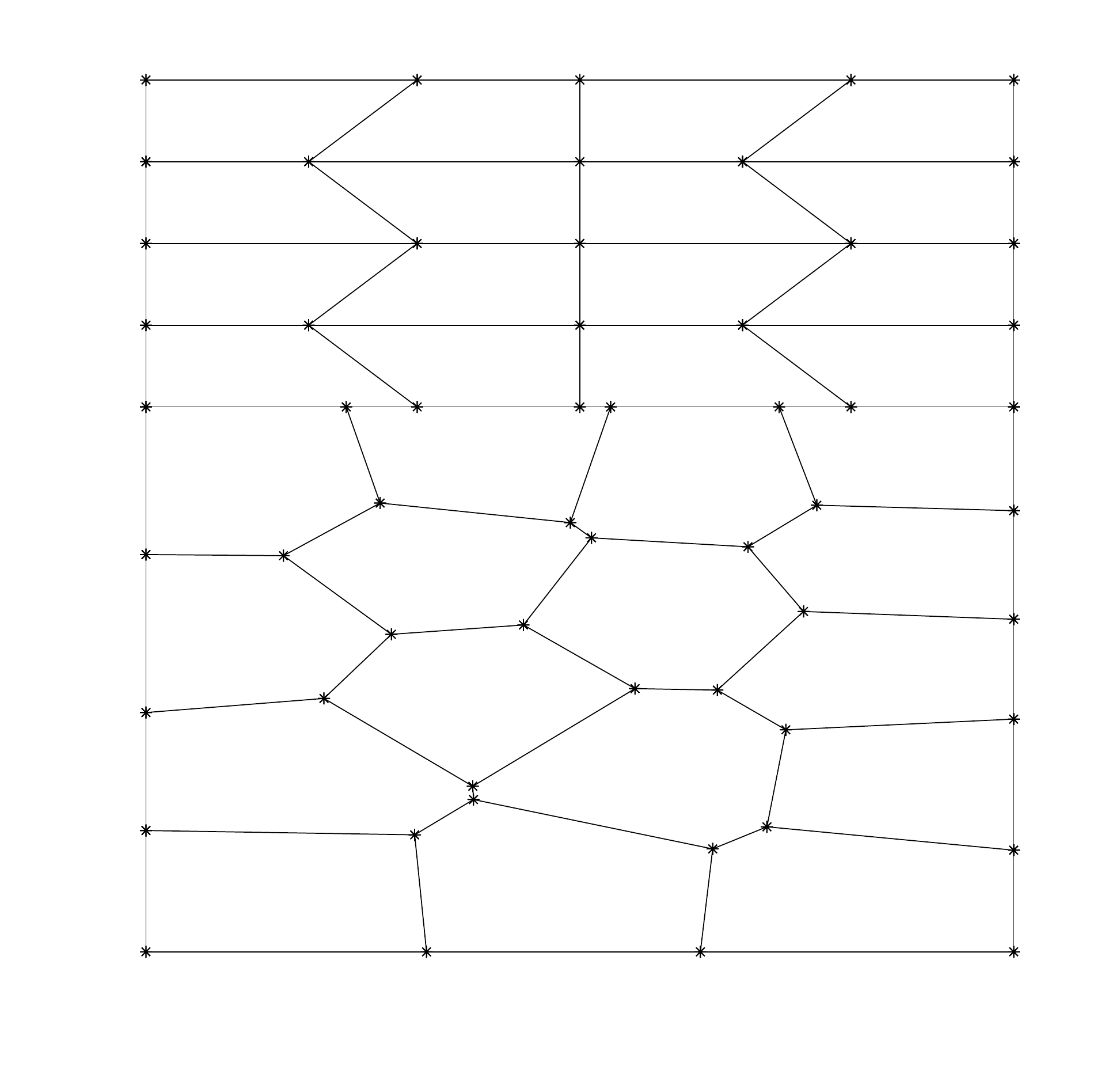}
%\centering\includegraphics[height=4.5cm, width=4.5cm]{tablasymallas/thao5.pdf}
%\centering\includegraphics[height=4.5cm, width=4.5cm]{tablasymallas/thao3T.pdf}
\end{minipage}
\caption{Sample meshes with small edges. From left to right:
$\CT_h^1$, $\CT_h^2$ for $N=4$ and $\CT_h^3$ for $N=4$.}
\label{FIG:meshes}
\end{center}
\end{figure}

\subsection{The load problem}
On this test our interest is to approximate the solution of \eqref{eq:weak_state_equation} for a given load $f\in H^1(\Omega)$. The aim of this test is to compute error rates in the corresponding norms of the problem. In this test, we consider two scenarios: the first one is considering the coefficient $\kappa(\mathbf{x})$ constant in the whole domain $\Omega$ and hence, piecewise constant on each polygon of the mesh, whereas in the second  we consider $\kappa(\mathbf{x})$ beyond our developed theory, where this coefficient is a bounded
function. In both tests, the domain  in which we state problem \eqref{eq:state_equation} is  the unit square $\Omega:=(0,1)^2$ with null Dirichlet boundary condition on $\partial\Omega$.
\subsubsection{Piecewise constant $\kappa(\mathbf{x})$}
Here, the value of  $\kappa(\mathbf{x})=1= \gamma(\mathbf{x})$, for all  $\mathbf{x}\in\Omega$
and
$\mathbf{\vartheta}(\mathbf{x}):=(x,y)^{\texttt{t}}$.
The  load  $f$ and Dirichlet boundary conditions are chosen in such a way
that the exact solution is $u(\mathbf{x}):=\sin(\pi x)\sin(\pi y).$

In Figure \ref{FIG:errorkappa1} we  show, in log-log scale, the error convergence curves in $L^2$ and $H^1$ between the solution $u$ and the polynomial projection of the virtual solution $u_h$.
\begin{figure}[H]
\begin{center}
\begin{minipage}{13cm}
\centering\includegraphics[height=5.0cm, width=4.0cm]{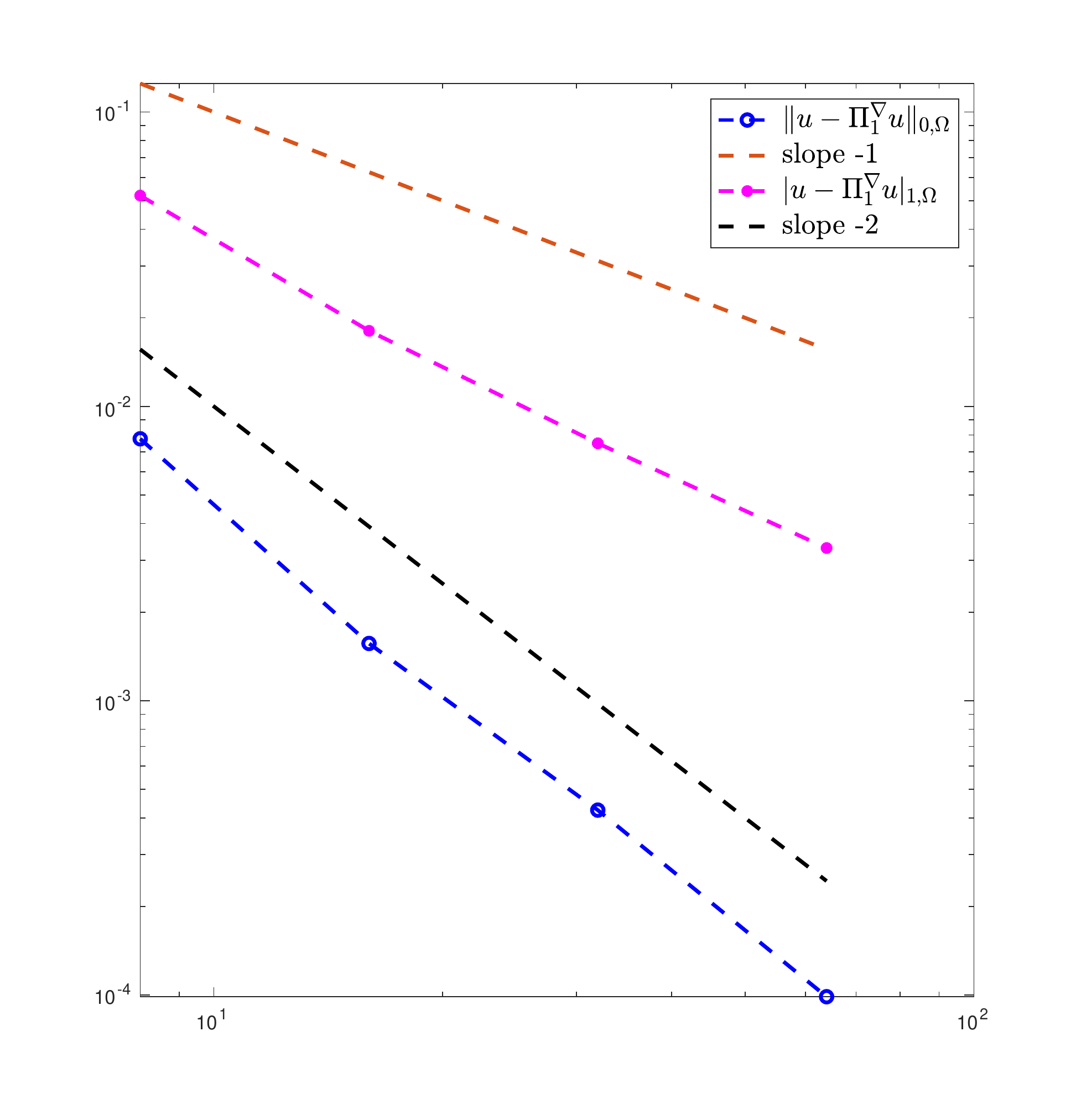}
\centering\includegraphics[height=5.0cm, width=4.0cm]{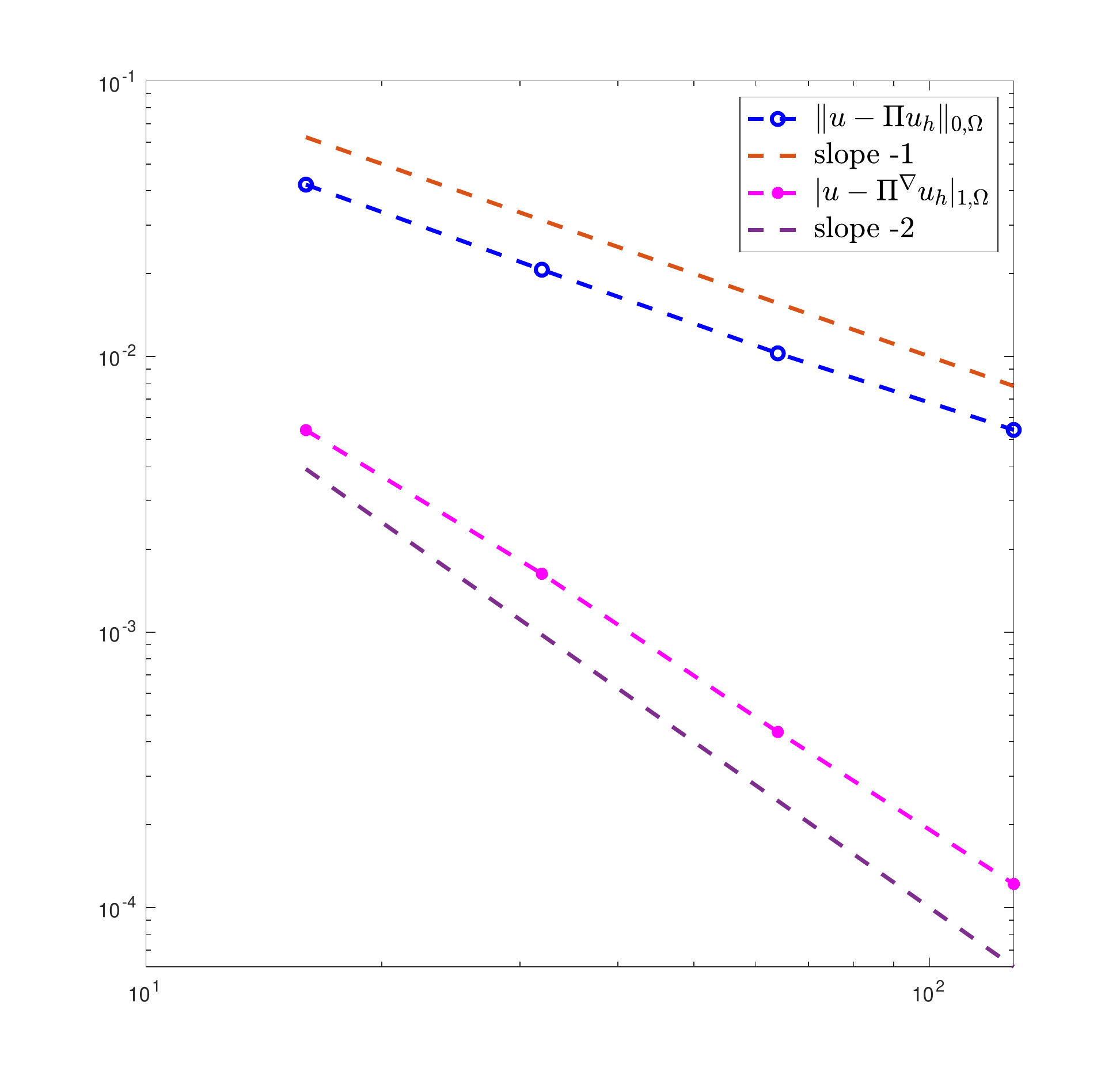}
\centering\includegraphics[height=5.0cm, width=4.0cm]{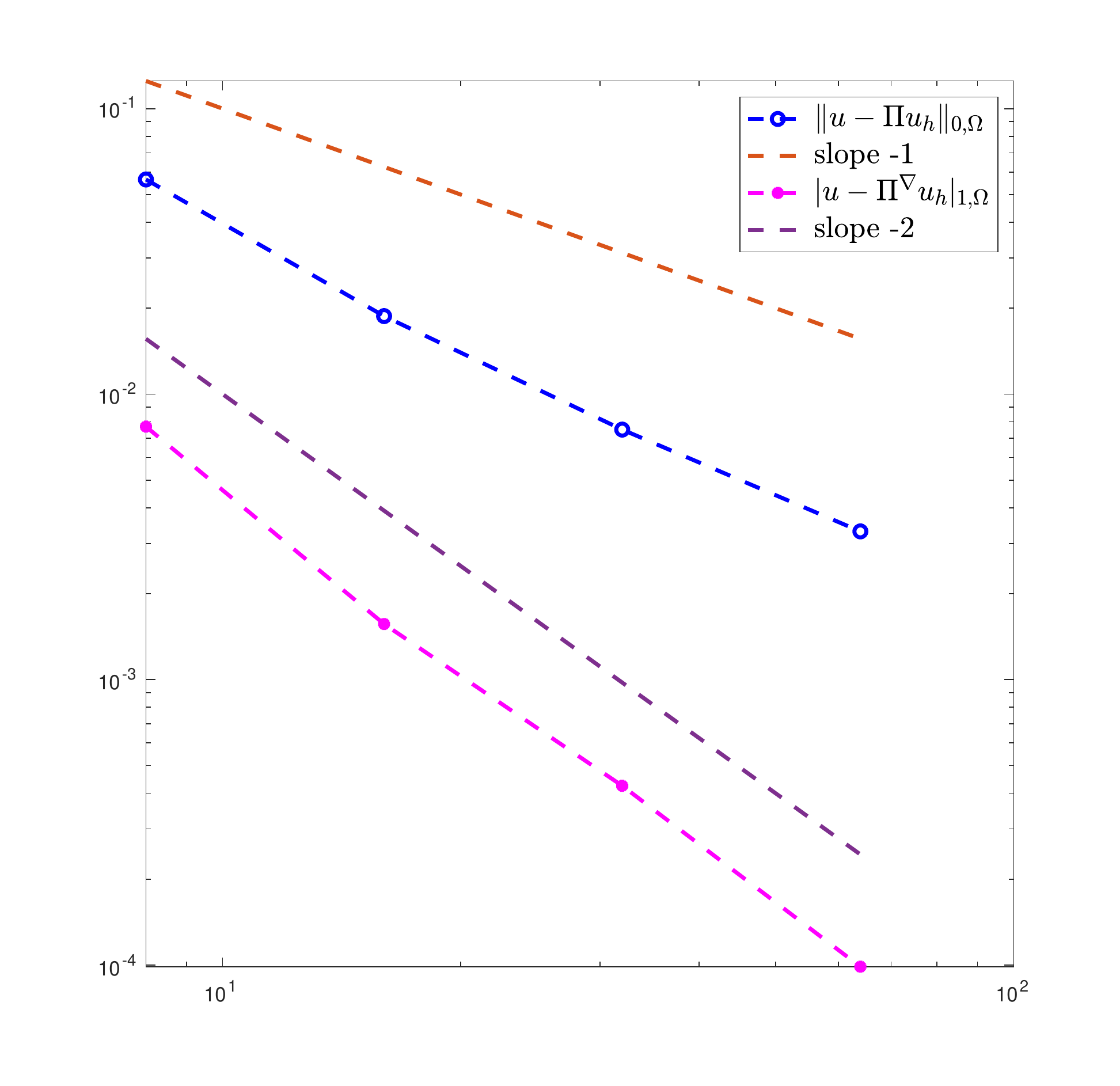}
\end{minipage}
\caption{Test 1: $L^2$ error with seminorm $H^1$ error for $\CT_h^1$ (left), $\CT_h^2$ (middle)and $\CT_h^3$ (right).}
\label{FIG:errorkappa1}
\end{center}
\end{figure}

We observe from Figure  \ref{FIG:errorkappa1} a clear quadratic order of convergence for error in $L^2$ and order 1 for the seminorm. Note that this is the optimal order, as demonstrated in the theoretical section.
\subsubsection{Bounded and smooth $\kappa(\mathbf{x})$}
In this test, we consider the following  functions
\[\kappa(\mathbf{x}):=1,%\left[ \begin{array}{cc}
%y^2+1 &-xy\\
%-xy & x^2+1
%\end{array} \right],
\qquad \mathbf{\vartheta}(\mathbf{x}):=\left( \begin{array}{c}
x\\
y
\end{array} \right),\qquad \text{ and }\qquad \gamma(\mathbf{x}):=x^2+y^3,\]
and with right-hand side and Dirichlet boundary conditions defined in such a way
that the exact solution is
$$u(\mathbf{x}):=(x-x^2)(y-y^2)+\sin(2\pi x)\sin(2\pi y).$$

Let us remark that the present test goes beyond of the developed theory, since for all the calculations hold, the hypothesis on the coefficients
is that all of the must be piecewise constant on each polygon, which in   Test 1 hold. However, the method is robust
independently of this assumption. In Figure \ref{FIG:error_no_constant} we observe precisely what we claim, where the error decreases for each of the meshes considered. 

\begin{figure}[H]
\begin{center}
\begin{minipage}{13cm}
\centering\includegraphics[height=5.0cm, width=4.0cm]{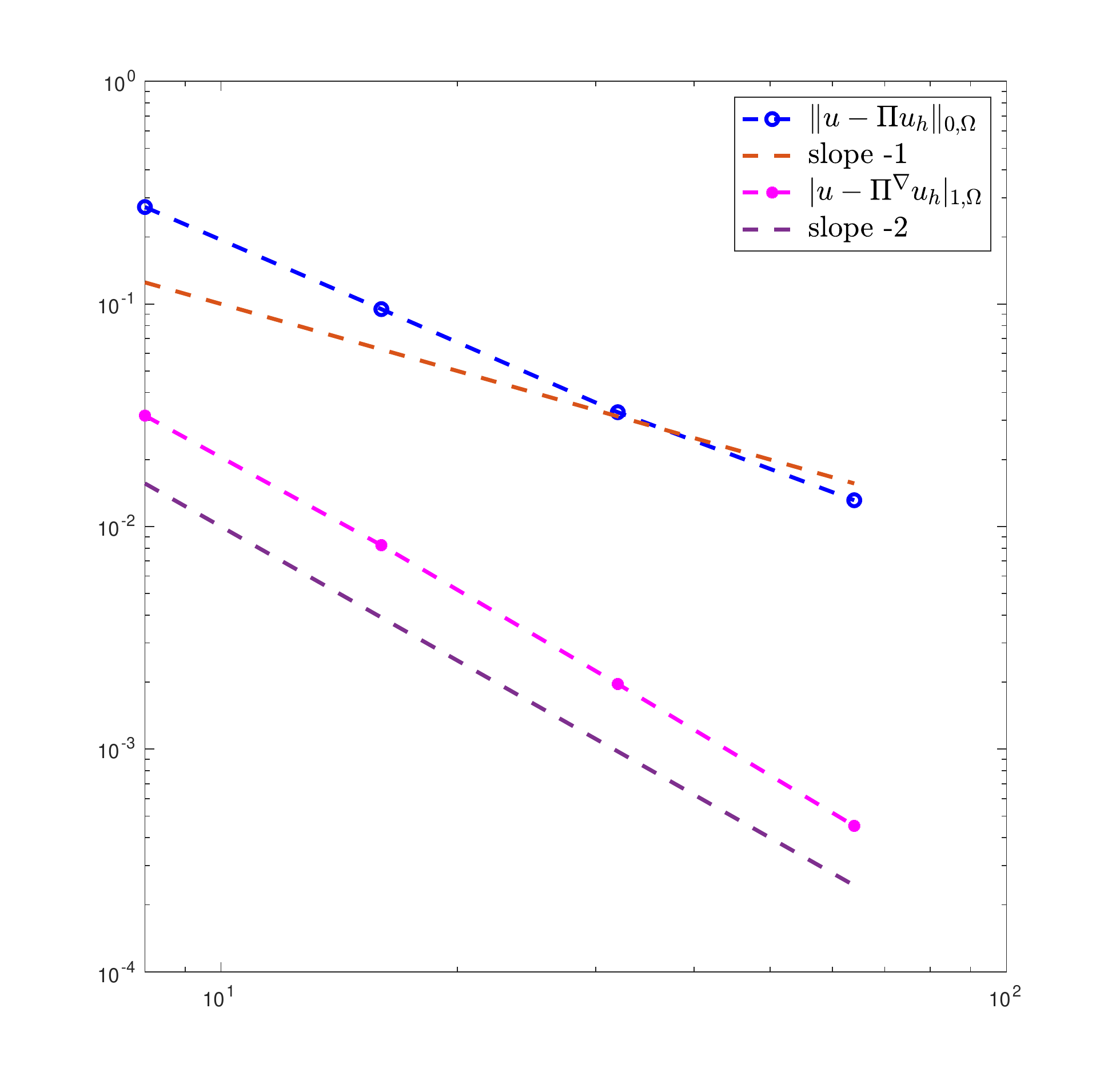}
\centering\includegraphics[height=5.0cm, width=4.0cm]{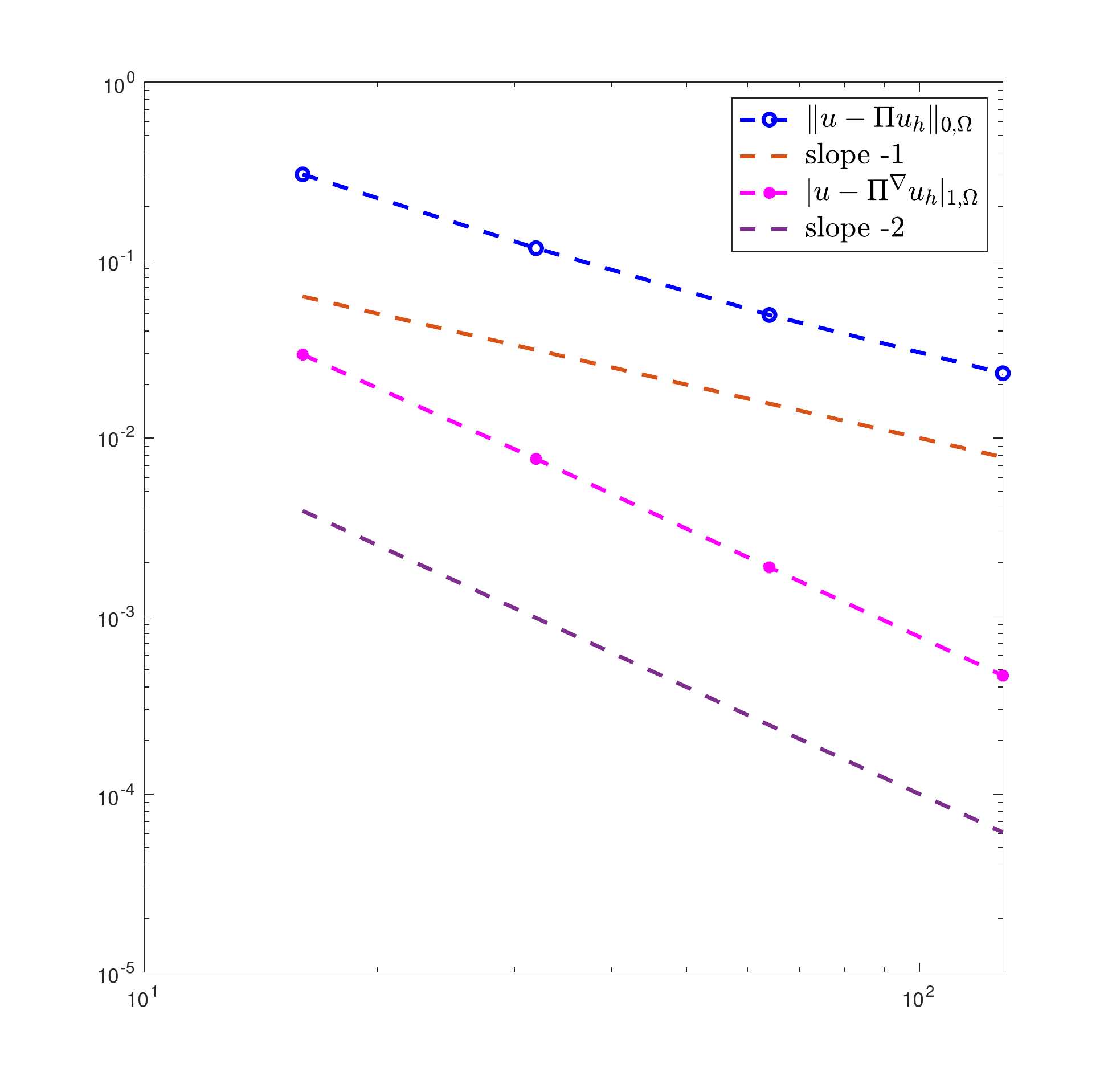}
\centering\includegraphics[height=5.0cm, width=4.0cm]{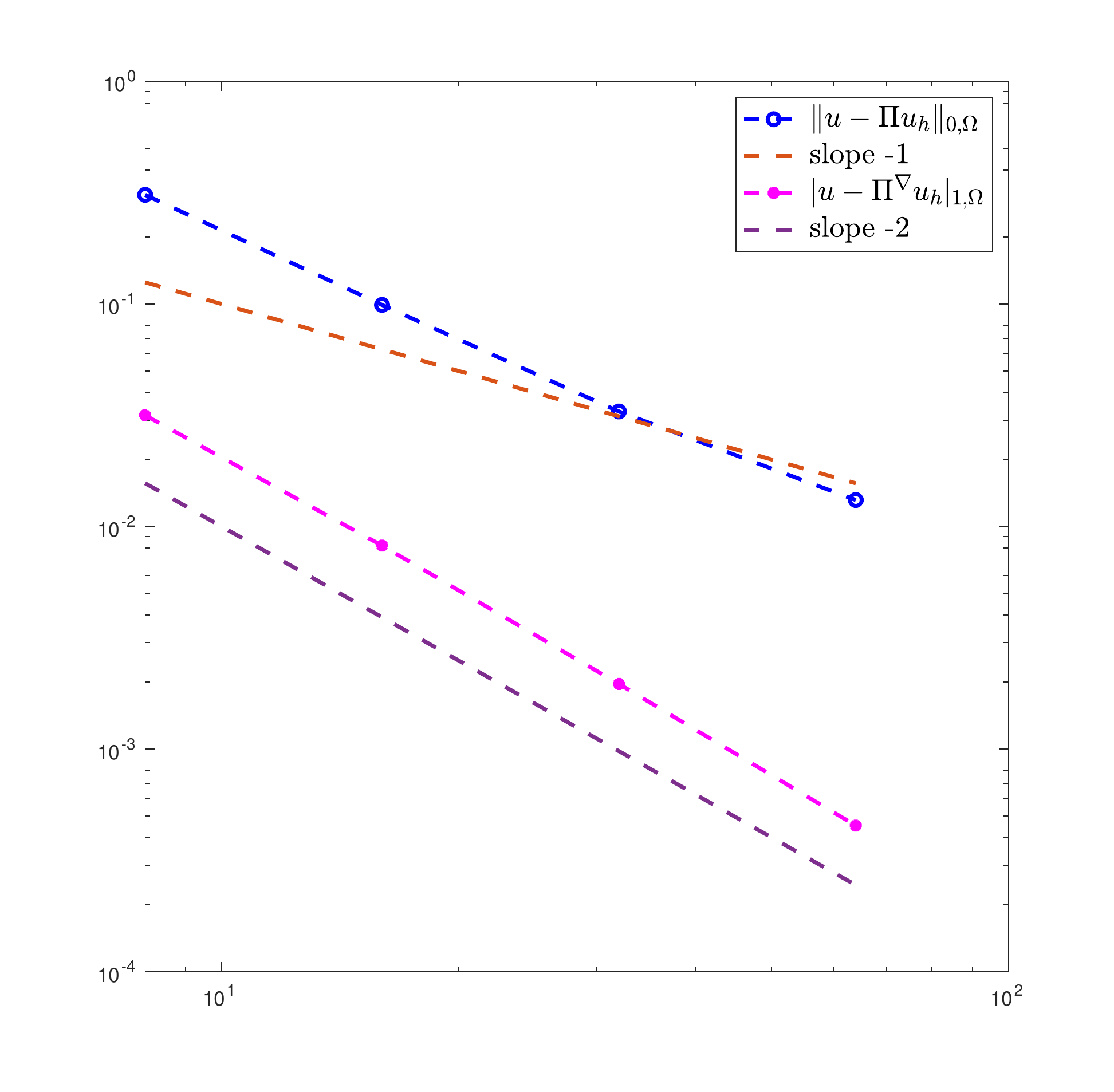}
\end{minipage}
\caption{Test 2: $L^2$ error with seminorm 1 error for $\CT_h^1$ (left), $\CT_h^2$ (middle)and $\CT_h^3$ (right).}
\label{FIG:error_no_constant}
\end{center}
\end{figure}

\subsection{The eigenvalue problem}
Now our task is to approximate the eigenvalues and eigenfunctions of problem \eqref{eq:spectral1} with our small edges approach. For this test we consider two scenarios: a convex and a non-convex domain. The order of convergence is computed with a least-square fitting.

\subsubsection{Unitary square}
Let us consider as computational domain 
the unit square $\O:=(0,1)^2$. This domain is discretized with meshes   presented in Figure~\ref{FIG:meshes}. The analytical solution to the convection-diffusion spectral problem is as follows (see \cite{MR2888313})
\begin{align}\nonumber
\lambda_{p,q}&=\dfrac{|\mathbf{\vartheta}|^2}{4\kappa}+\kappa\pi^2(p^2+q^2)=\lambda_{p,q}^*\qquad \text{for }p,q\in \mathbb{N}_+,\\
u_{p,q}(x,y)&=\exp\left(\dfrac{\mathbf{\vartheta}\cdot(x,y)}{2\kappa}\right)\sin(p\pi x)\sin(q\pi y),\\\label{eq:eigenex}
u_{p,q}^*(x,y)&=\exp\left(-\dfrac{\mathbf{\vartheta}\cdot(x,y)}{2\kappa}\right)\sin(p\pi x)\sin(q\pi y).\\\nonumber
\end{align}
For this test, we have used $\mathbf{\vartheta}(\mathbf{x})=(1,0)^t$ and $\kappa(\mathbf{x})=1$. In Table \ref{Table:1_spectral}, we report the first six eigenvalues computed with meshes $\CT_h^2$ and $\CT_h^3$. The row
``Order'' reports the convergence order of the eigenvalues, computed with respect to the exact ones obtained with \eqref{eq:eigenex}, which are presented in the row ``Exact”.
\begin{table}[h]
\begin{center}
\caption{Test~1. The lowest computed eigenvalues
$\l_{h}^{(i)}$, $1\le i\le 6$ for different meshes.}
\vspace{0.3cm}
\begin{tabular}{|c|c|c|c|c|c|c|}
\hline
\multicolumn{7}{|c|}{$\CT_{h}^{1}$} \\\hline
 $N$ & $\l_{h}^{(1)}$ & $\l_{h}^{(2)}$ & $\l_{h}^{(3)}$ & 
 $\l_{h}^{(4)}$ & $\l_{h}^{(5)}$& $\l_{h}^{(6)}$  \\
 \hline
8 &20.8691  & 52.3774  & 57.8481 &  93.6945 & 106.3515 & 134.1838\\
   16  & 20.1847  & 51.5506 &  50.2396 &  82.4563 & 100.6677  &107.8481\\
   32 &  20.0379  & 50.0865 &  49.7503 &  79.9785 &  99.3560&  101.1318\\
   64 &  20.0010  & 49.6345 & 49.7195 &  79.3947  & 99.0459 &  99.5049\\\hline
  Order&   2.08  &  2.08 &   2.03 &   2.09  &  2.07 &   2.00\\\hline
 Exact&   19.9892 &  49.5980 &  49.5980   &79.2068  & 98.9460 &  98.9460\\
\hline
\multicolumn{7}{|c|}{$\CT_{h}^{2}$} \\\hline
 $N$ & $\l_{h}^{(1)}$ & $\l_{h}^{(2)}$ & $\l_{h}^{(3)}$ & 
 $\l_{h}^{(4)}$ & $\l_{h}^{(5)}$& $\l_{h}^{(6)}$  \\
 \hline
8 &   20.8967 &  56.0401 &  56.1475 &  96.3325  & 125.4921 &  125.0578\\
16 &  20.2310 &  51.1774  & 51.1054 &  83.2877  & 105.0248 &  105.2693\\
32  & 20.0531 &  49.9872  & 49.9944 &  80.2748  & 100.5110  & 100.5074\\
64  & 20.0057  & 49.7015  & 49.7027  & 79.4760  & 99.3675   & 99.359\\\hline
Order& 1.93      &      1.99&   1.98    &           1.99  & 1.99  & 1.99\\\hline
Exact&      19.9892 &          49.5980     &       49.5980    &         79.2068    &         98.9460     &        98.9460 \\
\hline
\multicolumn{7}{|c|}{$\CT_{h}^{3}$} \\\hline
 $N$ & $\l_{h}^{(1)}$ & $\l_{h}^{(2)}$ & $\l_{h}^{(3)}$ & 
 $\l_{h}^{(4)}$ & $\l_{h}^{(5)}$& $\l_{h}^{(6)}$  \\
 \hline
8 &  20.8669 &  58.1517&   52.3612  & 94.0519  & 106.2181& 137.1134\\
16 & 20.1847 &  51.5743&   50.2385 &  82.4877 &  100.6575&  108.0124\\
32 & 20.0379 &  50.0877 &  49.7502 &  79.9805 &  99.3554   &101.1401\\
64 & 20.0010&  49.7196  &49.6345 & 79.3948 &  99.0459 &99.5054\\\hline
Order&2.07&   2.04 &  2.08  & 2.10 &  2.06  & 2.03\\\hline
Exact&19.9892        &   49.5980     &       49.5980    &         79.2068         &    98.9460    &         98.9460\\
\hline
\end{tabular}
\label{Table:1_spectral}
\end{center}
\end{table}
We observe from the reported results that for all  meshes, the eigenvalues converge with order $\mathcal{O}(h^2)$. Despite to the fact that the problem is non symmetric, the computed eigenvalues are all real. On the other hand, the eigenvalues converge to the exact eigenvalues that we show in the row "Exact". 

To end this test, in Figure \ref{FIG:eigenfunctions}
we present plots of the first, second and four eigenfunctions of our problem. %Clearly since the third and fourth eigenvalue are the same (multiplicity two), the eigenfunctions are symmetric.

\begin{figure}[h]
\begin{center}
\begin{minipage}{13cm}
\centering\includegraphics[height=4.cm, width=4cm]{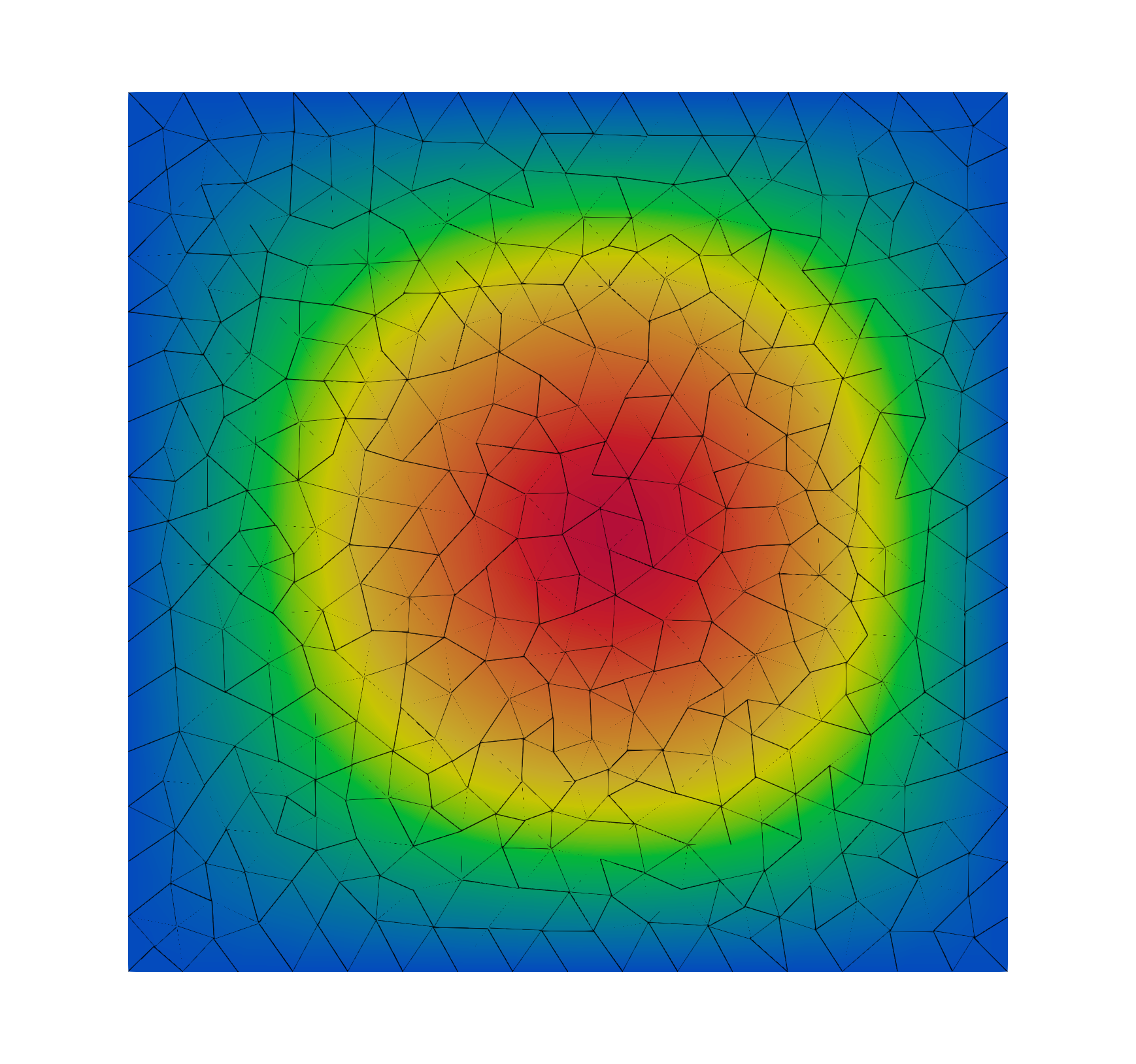}
\centering\includegraphics[height=4cm, width=4cm]{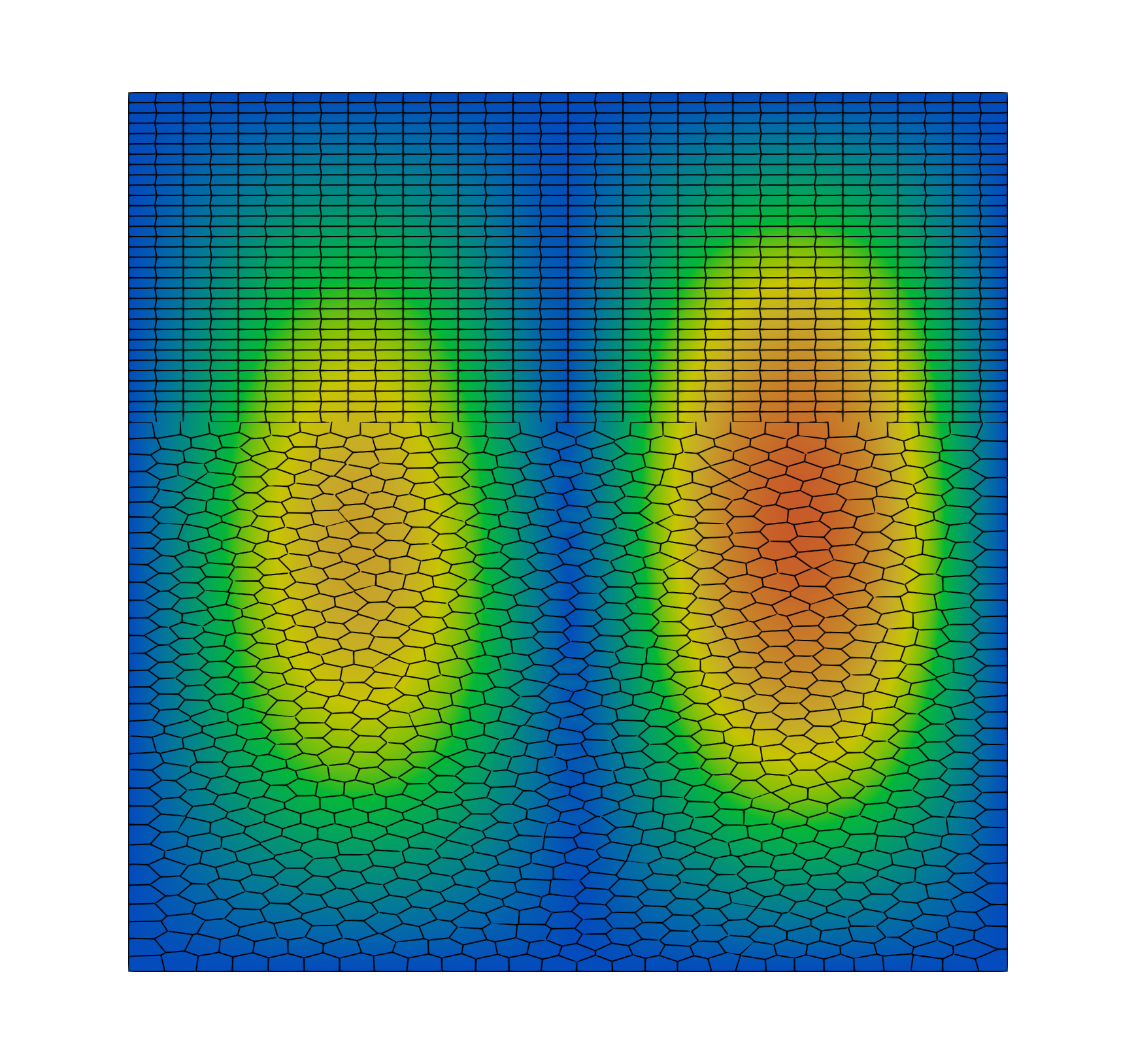}
\centering\includegraphics[height=4cm, width=4cm]{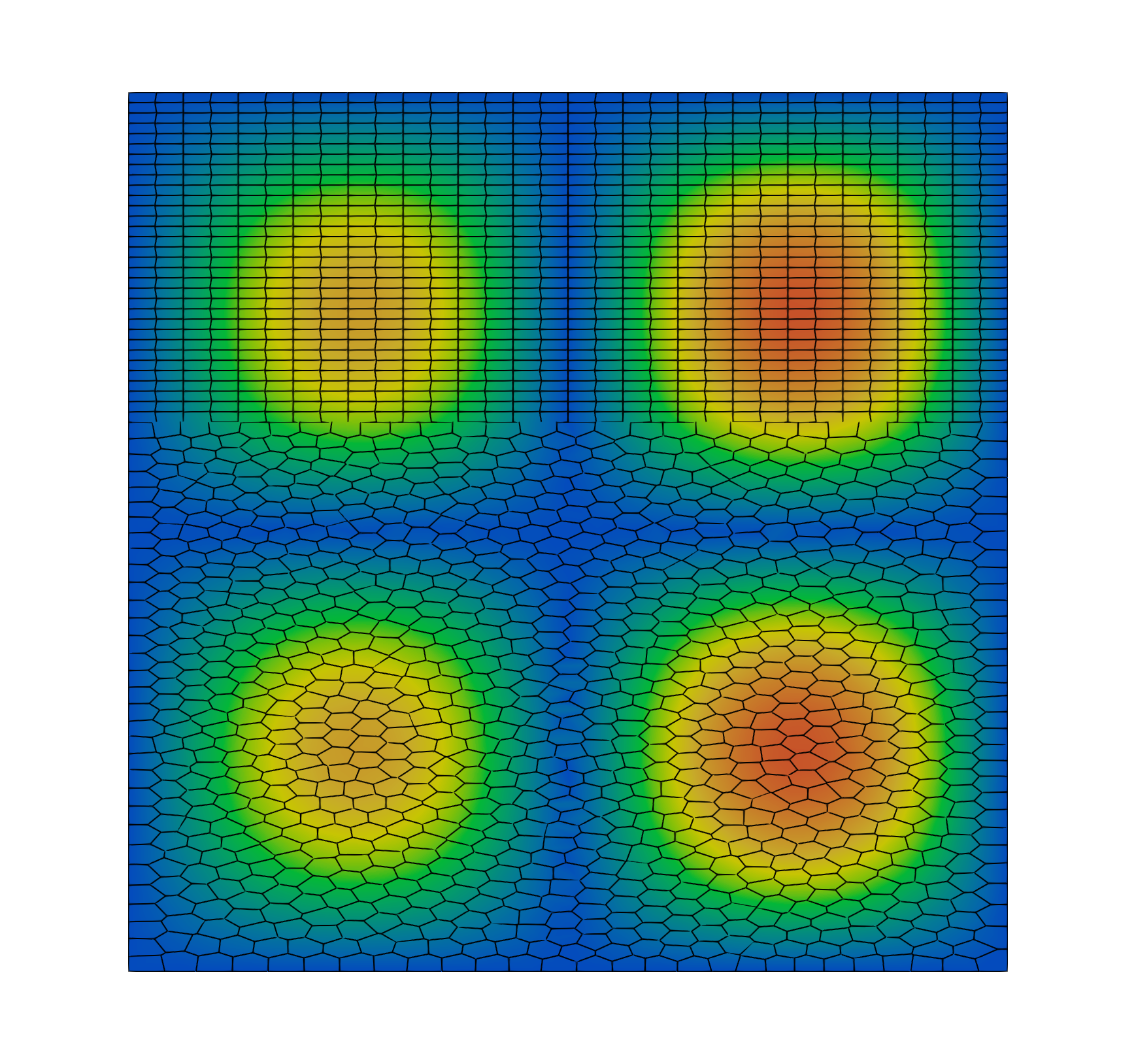}
%   \centering\includegraphics[height=4.5cm, width=4.5cm]{tablasymallas/VorVor.pdf}
%   \centering\includegraphics[height=4.5cm, width=4.5cm]{tablasymallas/thao5.pdf}
% \centering\includegraphics[height=4.5cm, width=4.5cm]{tablasymallas/thao3T.pdf}
\end{minipage}
\caption{From left to  right, plots of the first, second and four eigenfunctions, computed with  $\CT_h^2$ and $\CT_h^3$.}
\label{FIG:eigenfunctions}
\end{center}
\end{figure}

\subsubsection{Non convex domain}

In the following test we will consider a non-convex
domain which we call \emph{rotated T}, and it is defined
by $\Omega:=(-0.5, 0.5)\times(-0.5, 0)\cup (-0.25,0.25)\times(0,1)$
with boundary condition $u=0$ on the whole boundary $\partial\Omega$.
This non-convex domain presents two reentrant angles of the same size
$\omega=3\pi/2$ (cf. Figure~\ref{FIG:meshes2}),
and as a consequence, the eigenfunctions
of this problem may present singularities.

In Figure~\ref{FIG:meshes2}, we present the meshes
that we will consider for this numerical test.
We note that the families of polygonal meshes $\CT_h^3$, $\CT_h^4$ and $\CT_h^5$
have been obtained by gluing two different polygonal meshes at $x=0$.
It can be seen that very small edges compared with the element
diameter appears on the interface of the resulting meshes.
\begin{figure}[h]
\begin{center}
\begin{minipage}{13cm}
\centering\includegraphics[height=4.0cm, width=4.0cm]{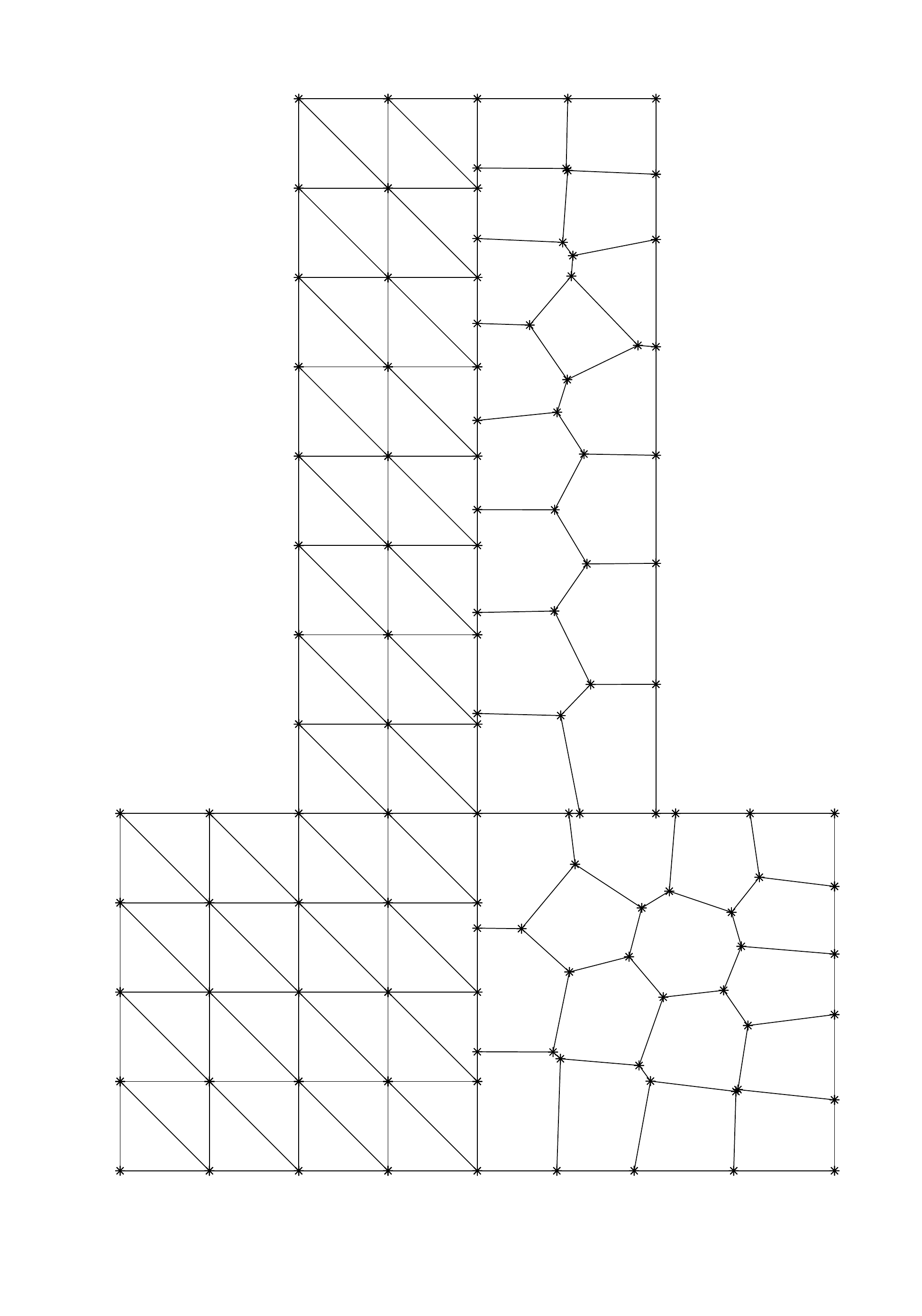}
\centering\includegraphics[height=4.0cm, width=4.0cm]{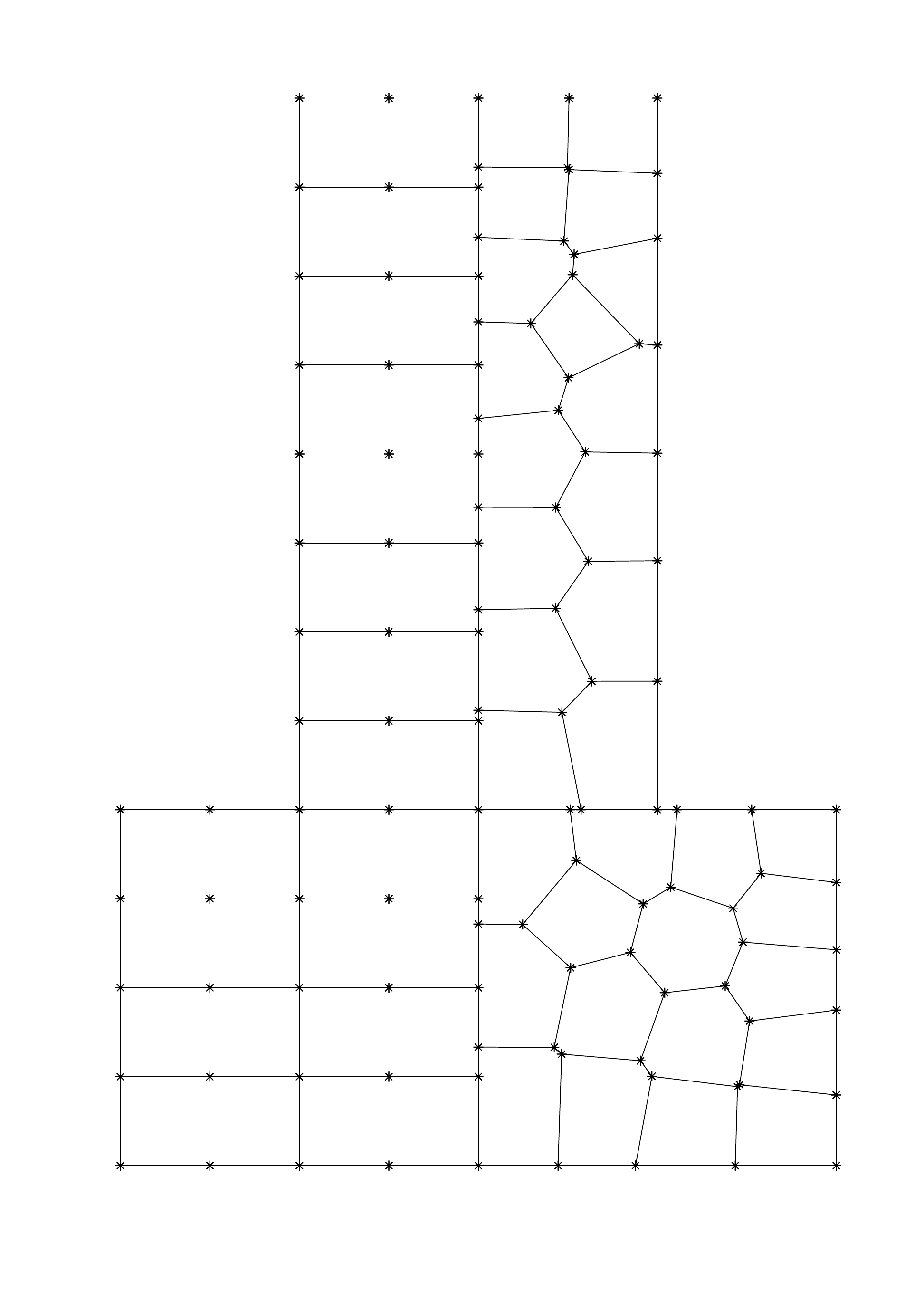}\\
\centering\includegraphics[height=4.0cm, width=4.0cm]{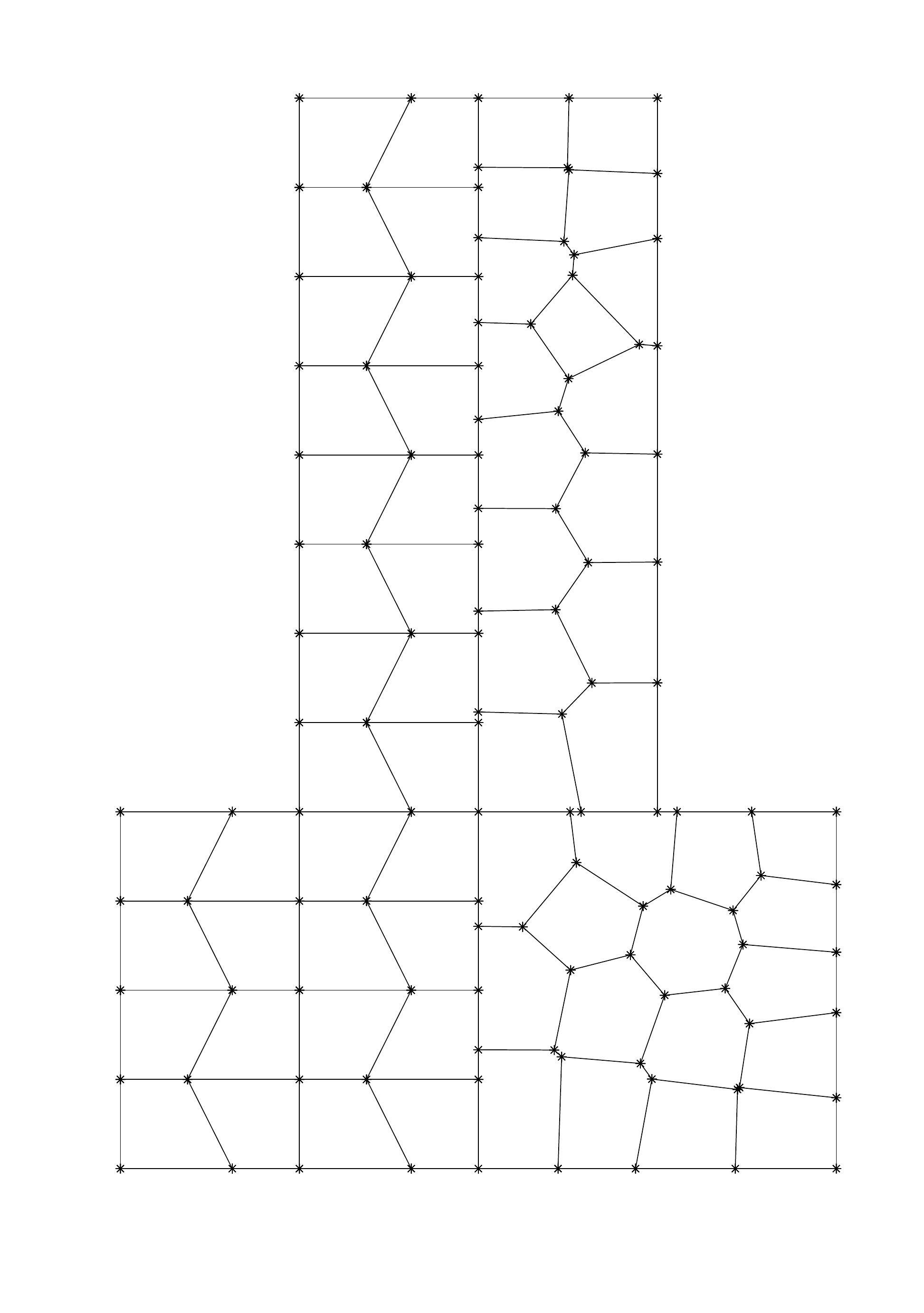}
  \centering\includegraphics[height=4.0cm, width=4.0cm]{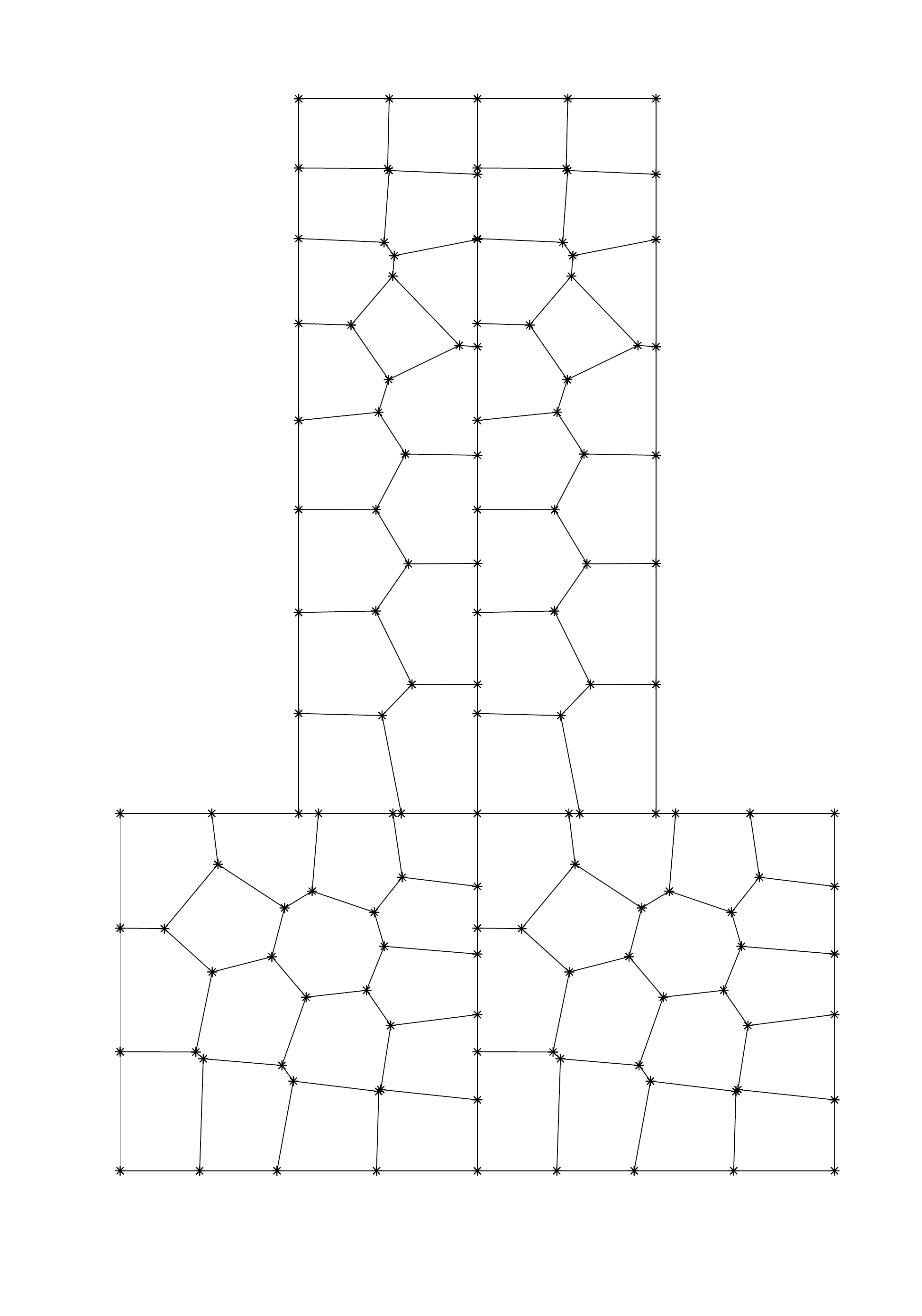}
%   \centering\includegraphics[height=4.5cm, width=4.5cm]{tablasymallas/thao5.pdf}
% \centering\includegraphics[height=4.5cm, width=4.5cm]{tablasymallas/thao3T.pdf}
\end{minipage}
\caption{Sample meshes with small edges. From top left to bottom right: $\CT_h^4$, $\CT_h^5$, $\CT_h^6$ and $\CT_h^7$, for $N=8$.}
\label{FIG:meshes2}
\end{center}
\end{figure}

In Table \ref{Table:T_spectral} we report the first six computed eigenvalues with our method. We claim that for this geometry we do not have
analytical solution. Hence, we compare our results with extrapolated values that we present in the row "Extrap". As in the previous example, the order of convergence, reported in the row "Order" have been computed with a least-square fitting.

\begin{table}[h]
\begin{center}
\caption{Test~1. The lowest computed eigenvalues
$\l_{h}^{(i)}$, $1\le i\le 6$ for different meshes.}
\vspace{0.3cm}
\begin{tabular}{|c|c|c|c|c|c|c|}
\hline
\multicolumn{7}{|c|}{$\CT_{h}^{4}$} \\\hline
 $N$ & $\l_{h}^{(1)}$ & $\l_{h}^{(2)}$ & $\l_{h}^{(3)}$ & 
 $\l_{h}^{(4)}$ & $\l_{h}^{(5)}$& $\l_{h}^{(6)}$  \\
 \hline
   16 &  35.8647&   50.9369&   74.3495  & 79.8639&  107.1151&  138.1449\\
   30  & 34.9179  & 49.8908  & 72.0953  & 76.7690  &102.5379  &130.5160\\
   62   &34.5028 &  49.6345 &  71.5289  & 75.8149 & 101.3815  &128.6860\\
  130  & 34.3804  & 49.5724  & 71.3967 &  75.5511 & 101.0999 & 128.2229\\\hline
  Order&   1.50  &  2.28    &2.25  & 1.95    &2.24   & 2.09\\\hline
    Extrap.&   34.3074 &  49.5656  & 71.3771&   75.4885 & 101.0644 & 127.9790\\
\hline
\multicolumn{7}{|c|}{$\CT_{h}^{5}$} \\\hline
 $N$ & $\l_{h}^{(1)}$ & $\l_{h}^{(2)}$ & $\l_{h}^{(3)}$ & 
 $\l_{h}^{(4)}$ & $\l_{h}^{(5)}$& $\l_{h}^{(6)}$  \\
 \hline
16 & 36.1795 &  51.3140&   75.4220  & 81.3284 & 109.2599  &141.7893\\
   30&   35.0277 &  49.9852&   72.3584&   77.1608 & 103.0639  &131.3561\\
   62  & 34.5422  & 49.6585  & 71.5953 &  75.9265 & 101.5153 & 128.8964\\
  130   &34.3950  & 49.5786   &71.4136  & 75.5848&  101.1348 & 128.2769\\\hline
  Order&   1.54  &  2.27   & 2.26   & 2.00   & 2.25   & 2.32\\\hline
    Extrap.&  34.3172 &  49.5695 &  71.3900&   75.5147  &101.0894  &128.2296\\
\hline
\multicolumn{7}{|c|}{$\CT_{h}^{6}$} \\\hline
 $N$ & $\l_{h}^{(1)}$ & $\l_{h}^{(2)}$ & $\l_{h}^{(3)}$ & 
 $\l_{h}^{(4)}$ & $\l_{h}^{(5)}$& $\l_{h}^{(6)}$  \\
 \hline
   16  & 36.1709 &  51.3411  & 75.4263 &  81.3049 & 109.2335 & 141.8461\\
   30  & 35.0258 &  49.9870&   72.3587&   77.1569 & 103.0624 & 131.3610\\
   62   &34.5418 &  49.6586 &  71.5954  & 75.9258  &101.5152  &128.8966\\
  130  & 34.3949  & 49.5786  & 71.4136  & 75.5846  &101.1348 & 128.2769\\\hline
 Order&   1.54   & 2.29   & 2.26   & 1.99  & 2.24   & 2.33\\\hline
    Extrap.&     34.3180 &  49.5702  & 71.3897 &  75.5101  &101.0851  &128.2359\\
\hline
\multicolumn{7}{|c|}{$\CT_{h}^{7}$} \\\hline
 $N$ & $\l_{h}^{(1)}$ & $\l_{h}^{(2)}$ & $\l_{h}^{(3)}$ & 
 $\l_{h}^{(4)}$ & $\l_{h}^{(5)}$& $\l_{h}^{(6)}$  \\
 \hline
16 & 36.1532&   51.2579  & 75.0659  & 80.8267  &108.4547 & 140.4589\\
   28&  35.0908  & 49.9516   &72.2464  & 77.0487&  102.8661  &130.9399\\
   60 &  34.5461&  49.6508  & 71.5580  & 75.8867  &101.4574 & 128.7725\\
  132 &  34.3937 &  49.5760 &  71.4041&   75.5698&  101.1188 &128.2443\\\hline
  Order&  1.56  &  2.69   & 2.62   & 2.23  & 2.56   & 2.71\\\hline
    Extrap.&  34.3209  & 49.5847 &  71.4127  & 75.5577 & 101.1394 & 128.3053\\
\hline
\end{tabular}
\label{Table:T_spectral}
\end{center}
\end{table}

From Table \ref{Table:T_spectral} it is possible to observe the effects of the singularities of the domain on
the computed order of convergence for the first eigenvalue. Clearly the eigenfunction associated to this 
eigenvalue is non smooth, which precisely affects the order of convergence. We remark that this phenomenon 
occurs for each of the meshes considered for this test.  On the other hand, the rest of the computed eigenvalues converge to the 
extrapolated values with order $\mathcal{O}(h^2)$ as is expected.

Once again, the computed eigenvalues for this test are real and no complex eigenvalues have been observed.
We end our test presenting  plots for the first four eigenfunctions  of the spectral problem  in Figure \ref{FIG:eigenfunctionsT}.

\begin{figure}[h]
\begin{center}
\begin{minipage}{11cm}
\centering\includegraphics[height=4cm, width=4.2cm]{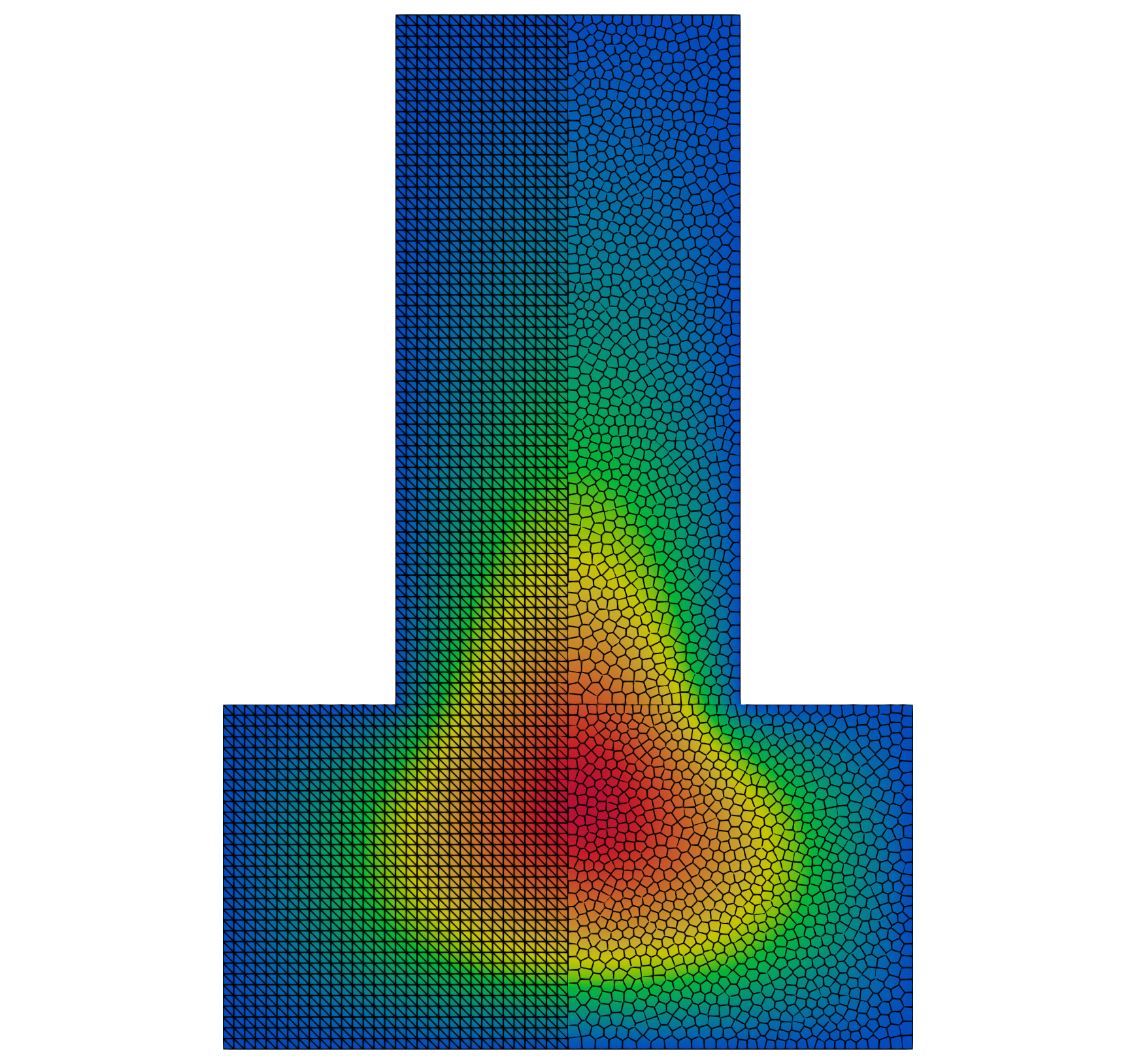}
\centering\includegraphics[height=4cm, width=4.2cm]{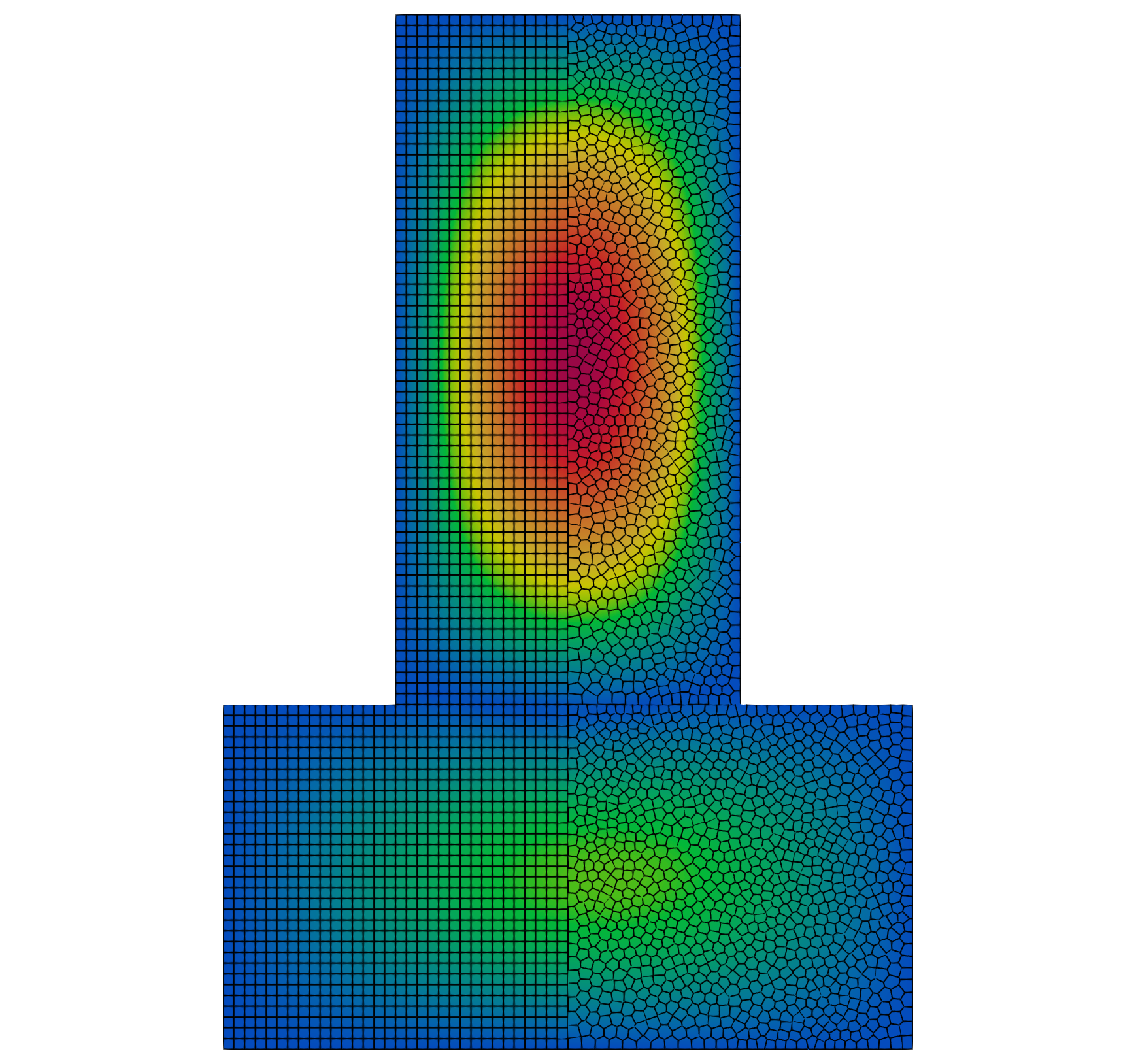}\\
\centering\includegraphics[height=4cm, width=4.2cm]{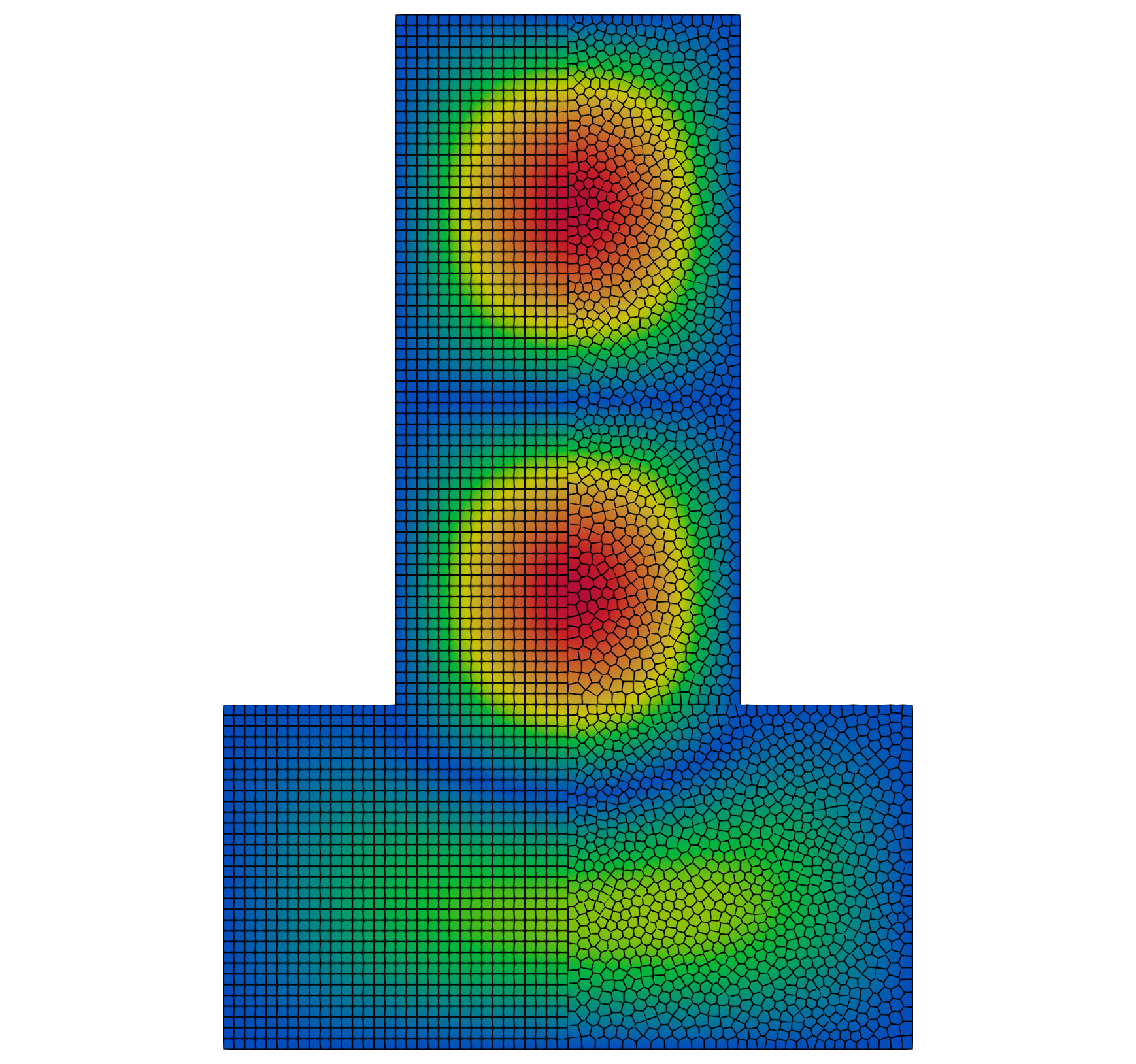}
\centering\includegraphics[height=4cm, width=4.2cm]{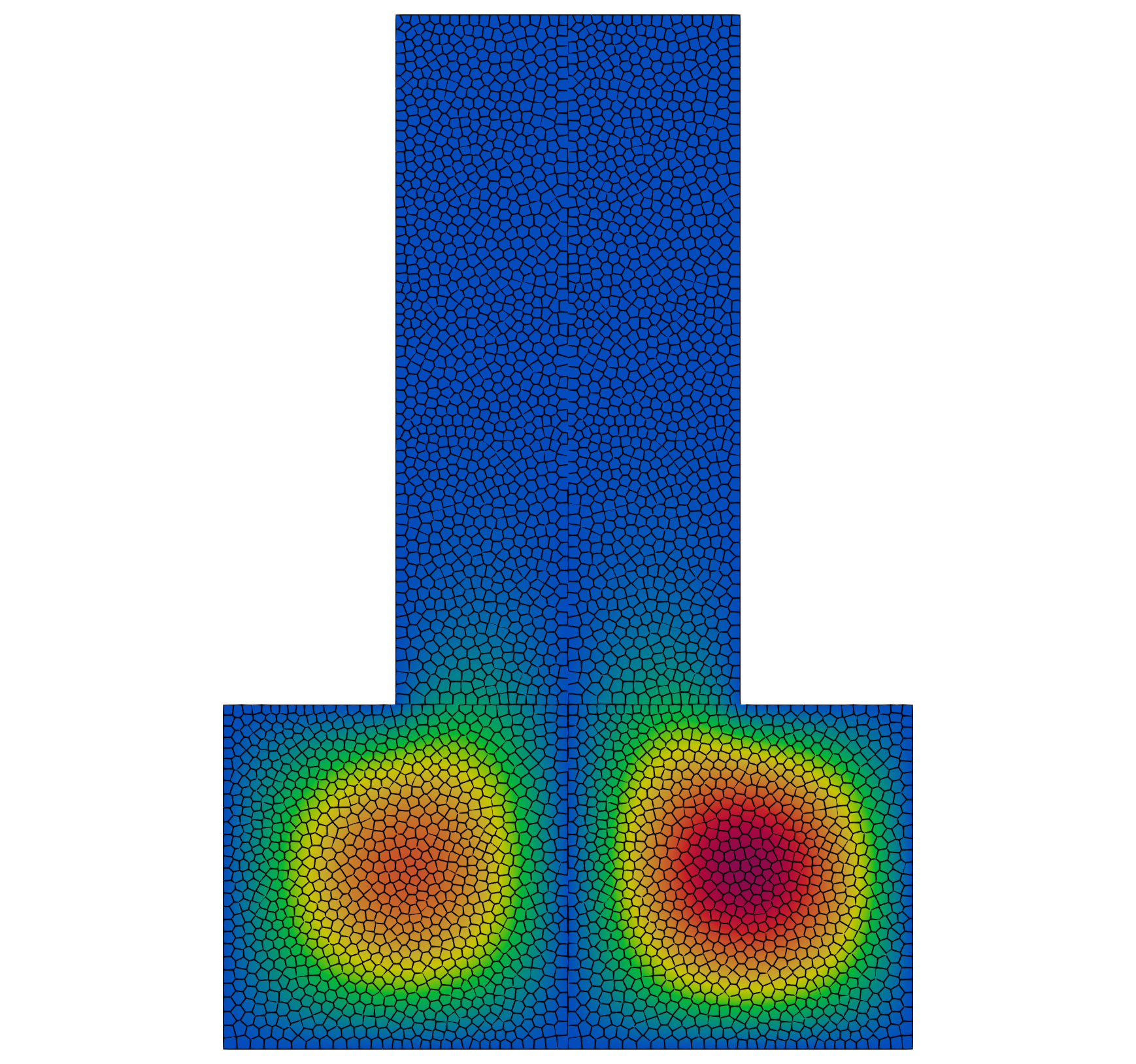}
%%   \centering\includegraphics[height=4.5cm, width=4.5cm]{tablasymallas/thao5.pdf}
%% \centering\includegraphics[height=4.5cm, width=4.5cm]{tablasymallas/thao3T.pdf}
\end{minipage}
\caption{From top left to bottom right, plots of the first four eigenfunctions for the rotated T domain, computed with  $\CT_h^4$, $\CT_h^5$, $\CT_h^6$ and $\CT_h^7$, respectively.}
\label{FIG:eigenfunctionsT}
\end{center}
\end{figure}

\bibliographystyle{siam} %siam, abbrv, ieeetr, unsrt, acm
\bibliography{LR3}

\begin{thebibliography}{10}

\bibitem{MR4441208}
{\sc D.~Adak, G.~Manzini, and S.~Natarajan}, {\em Virtual element approximation
  of two-dimensional parabolic variational inequalities}, Comput. Math. Appl.,
  116 (2022), pp.~48--70.

\bibitem{AABMR13}
{\sc B.~Ahmad, A.~Alsaedi, F.~Brezzi, L.~D. Marini, and A.~Russo}, {\em
  Equivalent projectors for virtual element methods}, Comput. Math. Appl., 66
  (2013), pp.~376--391.

\bibitem{ALR:22}
{\sc D.~Amigo, F.~Lepe, and G.~Rivera}, {\em A virtual element method for the
  elasticity problem allowing small edges}, arXiv:2211.02792,  (2022).

\bibitem{ABM2022}
{\sc P.~F. Antonietti, L.~Beir\~{a}o~da Veiga, and G.~Manzini}, {\em The
  Virtual Element Method and its Applications}, vol.~31, SEMA SIMAI Springer
  Series, 2022.

\bibitem{ABMV2014}
{\sc P.~F. Antonietti, L.~Beir\~{a}o~da Veiga, D.~Mora, and M.~Verani}, {\em A
  stream virtual element formulation of the {S}tokes problem on polygonal
  meshes}, SIAM J. Numer. Anal., 52 (2014), pp.~386--404.

\bibitem{ABSVsinum16}
{\sc P.~F. Antonietti, L.~Beir\~{a}o~da Veiga, S.~Scacchi, and M.~Verani}, {\em
  A {$C^1$} virtual element method for the {C}ahn-{H}illiard equation with
  polygonal meshes}, SIAM J. Numer. Anal., 54 (2016), pp.~34--56.

\bibitem{AMSLP2018}
{\sc E.~Artioli, S.~de~Miranda, C.~Lovadina, and L.~Patruno}, {\em A family of
  virtual element methods for plane elasticity problems based on the
  {H}ellinger-{R}eissner principle}, Comput. Methods Appl. Mech. Engrg., 340
  (2018), pp.~978--999.

\bibitem{BO}
{\sc I.~Babu\v{s}ka and J.~Osborn}, {\em Eigenvalue problems}, Handb. Numer.
  Anal., II, North-Holland, Amsterdam, 1991.

\bibitem{BBCMMR2013}
{\sc L.~Beir\~{a}o~da Veiga, F.~Brezzi, A.~Cangiani, G.~Manzini, L.~D. Marini,
  and A.~Russo}, {\em Basic principles of virtual element methods}, Math.
  Models Methods Appl. Sci., 23 (2013), pp.~199--214.

\bibitem{BBDMR2017}
{\sc L.~Beir\~{a}o~da Veiga, F.~Brezzi, F.~Dassi, L.~D. Marini, and A.~Russo},
  {\em Virtual element approximation of 2{D} magnetostatic problems}, Comput.
  Methods Appl. Mech. Engrg., 327 (2017), pp.~173--195.

\bibitem{BBM}
{\sc L.~Beir\~{a}o~da Veiga, F.~Brezzi, and L.~D. Marini}, {\em Virtual
  elements for linear elasticity problems}, SIAM J. Numer. Anal., 51 (2013),
  pp.~794--812.

\bibitem{MR3460621}
{\sc L.~Beir\~{a}o~da Veiga, F.~Brezzi, L.~D. Marini, and A.~Russo}, {\em
  Virtual element method for general second-order elliptic problems on
  polygonal meshes}, Math. Models Methods Appl. Sci., 26 (2016), pp.~729--750.

\bibitem{BLR2017}
{\sc L.~Beir\~{a}o~da Veiga, C.~Lovadina, and A.~Russo}, {\em Stability
  analysis for the virtual element method}, Math. Models Methods Appl. Sci., 27
  (2017), pp.~2557--2594.

\bibitem{MR2652780}
{\sc D.~Boffi}, {\em Finite element approximation of eigenvalue problems}, Acta
  Numer., 19 (2010), pp.~1--120.

\bibitem{BS-2008}
{\sc S.~C. Brenner and L.~R. Scott}, {\em The Mathematical Theory of Finite
  Element Methods}, Springer, New York, 2008.

\bibitem{MR3815658}
{\sc S.~C. Brenner and L.Y. Sung}, {\em Virtual element methods on meshes with
  small edges or faces}, Math. Models Methods Appl. Sci., 28 (2018),
  pp.~1291--1336.

\bibitem{CG2017}
{\sc E.~C\'{a}ceres and G.~N. Gatica}, {\em A mixed virtual element method for
  the pseudostress-velocity formulation of the {S}tokes problem}, IMA J. Numer.
  Anal., 37 (2017), pp.~296--331.

\bibitem{CGPS}
{\sc A.~Cangiani, E.~H. Georgoulis, T.~Pryer, and O.~J. Sutton}, {\em A
  posteriori error estimates for the virtual element method}, Numer. Math., 137
  (2017), pp.~857--893.

\bibitem{MR2845628}
{\sc C.~Carstensen, J.~Gedicke, V.~Mehrmann, and A.~Miedlar}, {\em An adaptive
  homotopy approach for non-selfadjoint eigenvalue problems}, Numer. Math., 119
  (2011), pp.~557--583.

\bibitem{MR4359996}
{\sc J.~Droniou and L.~Yemm}, {\em Robust hybrid high-order method on polytopal
  meshes with small faces}, Comput. Methods Appl. Math., 22 (2022), pp.~47--71.

\bibitem{FS-M2AN18}
{\sc M.~Frittelli and I.~Sgura}, {\em Virtual element method for the
  {L}aplace-{B}eltrami equation on surfaces}, ESAIM Math. Model. Numer. Anal.,
  52 (2018), pp.~965--993.

\bibitem{GMV2018}
{\sc F.~Gardini, G.~Manzini, and G.~Vacca}, {\em The nonconforming virtual
  element method for eigenvalue problems}, ESAIM Math. Model. Numer. Anal., 53
  (2019), pp.~749--774.

\bibitem{GV:IMA2017}
{\sc F.~Gardini and G.~Vacca}, {\em Virtual element method for second-order
  elliptic eigenvalue problems}, IMA J. Numer. Anal., 38 (2018),
  pp.~2026--2054.

\bibitem{MR3133493}
{\sc J.~Gedicke and C.~Carstensen}, {\em A posteriori error estimators for
  convection-diffusion eigenvalue problems}, Comput. Methods Appl. Mech.
  Engrg., 268 (2014), pp.~160--177.

\bibitem{MR851383}
{\sc Vivette Girault and Pierre-Arnaud Raviart}, {\em Finite element methods
  for {N}avier-{S}tokes equations}, vol.~5 of Springer Series in Computational
  Mathematics, Springer-Verlag, Berlin, 1986.
\newblock Theory and algorithms.

\bibitem{MR0203473}
{\sc T.~Kato}, {\em Perturbation theory for linear operators}, Die Grundlehren
  der mathematischen Wissenschaften, Band 132, Springer-Verlag New York, Inc.,
  New York, 1966.

\bibitem{MR4284360}
{\sc F.~Lepe, D.~Mora, G.~Rivera, and I.~Vel\'{a}squez}, {\em A virtual element
  method for the {S}teklov eigenvalue problem allowing small edges}, J. Sci.
  Comput., 88 (2021), pp.~Paper No. 44, 21.

\bibitem{MR4253143}
{\sc F.~Lepe and G.~Rivera}, {\em A priori error analysis for a mixed {VEM}
  discretization of the spectral problem for the {L}aplacian operator},
  Calcolo, 58 (2021), pp.~Paper No. 20, 30.

\bibitem{MR4229296}
\leavevmode\vrule height 2pt depth -1.6pt width 23pt, {\em A virtual element
  approximation for the pseudostress formulation of the {S}tokes eigenvalue
  problem}, Comput. Methods Appl. Mech. Engrg., 379 (2021), pp.~Paper No.
  113753, 21.

\bibitem{MRR2015}
{\sc D.~Mora, G.~Rivera, and R.~Rodr\'{i}guez}, {\em A virtual element method
  for the {S}teklov eigenvalue problem}, Math. Models Methods Appl. Sci., 25
  (2015), pp.~1421--1445.

\bibitem{MR2888313}
{\sc A.~Naga and Z.~Zhang}, {\em Function value recovery and its application in
  eigenvalue problems}, SIAM J. Numer. Anal., 50 (2012), pp.~272--286.

\bibitem{MR4461634}
{\sc J.~Tushar, A.~Kumar, and S.~Kumar}, {\em Virtual element methods for
  general linear elliptic interface problems on polygonal meshes with small
  edges}, Comput. Math. Appl., 122 (2022), pp.~61--75.

\bibitem{WRR2016}
{\sc P.~Wriggers, W.~T. Rust, and B.~D. Reddy}, {\em A virtual element method
  for contact}, Comput. Mech., 58 (2016), pp.~1039--1050.

\end{thebibliography}
%\bibliographystyle{plain}
%\bibliography{biblio}
\end{document}